\documentclass[aop,preprint]{imsart}

\RequirePackage[colorlinks,citecolor=blue,urlcolor=blue]{hyperref}
\RequirePackage{amsthm,amsmath}
\RequirePackage[numbers]{natbib}
\usepackage{amssymb}
\usepackage{mathrsfs}
\usepackage[noend]{algpseudocode}

\newtheorem{theorem}{Theorem}[section]
\newtheorem{lemma}[theorem]{Lemma}
\newtheorem{corollary}[theorem]{Corollary}
\newtheorem{proposition}[theorem]{Proposition}
\newtheorem{question}[theorem]{Question}

\theoremstyle{definition}
\newtheorem{definition}[theorem]{Definition}
\newtheorem{example}[theorem]{Example}

\newtheorem{remark}[theorem]{Remark}

\numberwithin{equation}{section}

\newcommand{\notion}[1]{{\bf  \textit{#1}}}

\arxiv{arXiv:1809.01284}

\startlocaldefs
\endlocaldefs
\newcommand{\mycmdA}{\textnormal{Aut}(G)}
\begin{document}

\begin{frontmatter}

\title{Heavy Bernoulli-percolation clusters are indistinguishable}
\runtitle{Indistinguishability of Heavy Clusters}

\author{\fnms{Pengfei} \snm{Tang}\corref{}\ead[label=e1]{tangp@iu.edu}}
 \address{Department of Mathematics\\ Indiana University\\Bloomington, Indiana 47405-5701\\ USA\\ \printead{e1}}
\affiliation{Indiana University}


\runauthor{P. Tang}

\begin{abstract}
	We prove that  the heavy clusters  are indistinguishable  for  Bernoulli percolation on quasi-transitive nonunimodular graphs.  As an application, we show that the uniqueness threshold of any quasi-transitive graph is also the threshold for connectivity decay. This resolves a question of Lyons and Schramm (1999) in the Bernoulli percolation case and confirms a conjecture of Schonmann (2001). We also prove that every infinite cluster of Bernoulli percolation on a nonamenable quasi-transitive graph is transient almost surely. 
\end{abstract}


\end{frontmatter}

\section{Introduction}
Let $G=(V,E)$ be a connected, locally finite, quasi-transitive infinite graph and for simplicity we will just say “let $G$ be a quasi-transitive graph” hereafter. We allow  multiple edges and loops in $G$ but we will always assume $G$ is locally finite. Fix some parameter $p\in[0,1]$, consider Bernoulli($p$) percolation process on $G$.
The critical probability is defined as
\[
p_c=p_c(G):=\inf\{p\in[0,1]: \textnormal{a.s.\ there exists an infinite open cluster }   \}.
\]
Grimmett and Newman  \cite{GN1990} gave an example showing that there are some $p\in(0,1)$ such that for Bernoulli($p$) percolation on $T\times\mathbb{Z}$, a.s.\ there exist infinite many infinite clusters, where $T$ is a regular tree with high degree. Later Benjamini and Schramm \cite{BS1996} conjectured that if $G$ is a quasi-transitive nonamenable graph, then $p_c<p_u$, where the uniqueness threshold $p_u$ is defined as follows:
\[
p_u=p_u(G):=\inf\{p\in[0,1]:\textnormal{a.s.\ there exists a unique infinite open cluster }   \}.
\]

If $G$ is a quasi-transitive amenable graph, then there is at most one infinite cluster for Bernoulli percolation on $G$;
see \cite{BK1989} and \cite{GKN1992} for more details.  

Recently, for all quasi-transitive graphs whose automorphism group has a quasi-transitive nonunimodular  subgroup, the above conjecture has been proved by  Hutchcroft  \cite{Tom2017}.  The  conjecture has also been shown to hold for many  nonamenable unimodular graphs of special types. For details see the discussion in   \cite{Tom2017} and references therein.

Historically many properties for percolation processes on transitive graphs are first proved in the unimodular case \cite{BLPS1999b,HP1999,Timar2006a} while the nonunimodular case are proved later \cite{Tom2016,S1999} or remain open. The reason is that the mass transport principle \cite{BLPS1999} is a very  powerful technique in the unimodular case. One interesting fact about Hutchcroft's result in \cite{Tom2017}  is that he proved the above conjecture for  nonunimodular case first while general unimodular case remains open. The present paper also mainly focuses on nonunimodular quasi-transitive graphs.

If $G$ is a quasi-transitive graph with $p_c<p_u$ and $p\in (p_c,p_u)$, then Bernoulli($p$) percolation has infinitely many infinite open clusters. Are these infinite open clusters similar or different? Lyons and Schramm  \cite{LS1999a} showed for every graph $G$ with a transitive unimodular closed automorphism group $\Gamma\subset \mycmdA$, every $\Gamma$-invariant, insertion-tolerant percolation process on $G$ has indistinguishable infinite clusters. 

Suppose $G$ is a quasi-transitive graph and suppose $\Gamma\subset\textnormal{Aut}(G)$ is closed, nonunimodular and acts quasi-transitively on $G$. Let $m$ denote the Haar measure on $\Gamma$ ($m$ is unique up to a multiplicative constant). In particular, let $m(x):=m(\Gamma_x)$, where $\Gamma_x:=\{\gamma\in\Gamma:\gamma x=x \}$ denotes the stabilizer of $x\in V$. Then there are two types of infinite clusters: for an infinite cluster $C$, if $\sum_{x\in C}{m(x)}<\infty$, it is called a ($\Gamma$-)light cluster; otherwise it is called a ($\Gamma$-)heavy cluster. 
For the  nonunimodular case, Lyons and Schramm \cite{LS1999a} also pointed out that light clusters can be distinguished by $\Gamma$-invariant properties and they asked whether heavy clusters are indistinguishable (Question 3.17 there).

Here we give a positive answer and the exact definitions of indistinguishability  and $\Gamma$-invariant properties are given later in Definition \ref{def: indistinguishable infinite clusters} and Definition \ref{def: invariant property}. 
\begin{theorem}\label{thm:heavy clusters are indis}
	Suppose $G$ is a locally finite, connected infinite graph, and suppose that $\textnormal{Aut}(G)$ has  a closed, quasi-transitive and nonunimodular subgroup $\Gamma$. If there are infinite many $\Gamma$-heavy clusters a.s. for Bernoulli$(p)$ percolation on $G$, then they are  indistinguishable by $\Gamma$-invariant properties.  
\end{theorem}

Question 3.17 in \cite{LS1999a} was indeed asked for general insertion-tolerant percolation processes. Here we only have a positive answer for the Bernoulli percolation case. The general case is still open. For more discussion on the general case see the last section. \\

Let $d$ be the graph distance on $G$ and $B(x,N):=\{z\in V:d(z,x)\leq N\}$ be the ball centered at $x$ with radius $N$. Let $B(x,N)\leftrightarrow B(y,N)$ denote the event that there is an open path connecting some vertex $u\in B(x,N)$ and some vertex $v\in B(y,N)$. In particular we use $x\leftrightarrow y$ to denote the event that there is an open path connecting the two vertices $x,y$. 
Schonmann \cite{S1999} also proved a criterion of $p_u$ for all quasi-transitive graphs: 
\[
p_u=\inf\{p\in[0,1]: \lim_{N\rightarrow\infty} \inf_{x,y\in V}\mathbf{P}_p(B(x,N)\leftrightarrow B(y,N))=1   \}.
\]  

For unimodular transitive graphs, Lyons and Schramm \cite{LS1999a} gave another characterization of $p_u$: let $G$ be a unimodular transitive graph and $o\in V$ be a fixed vertex. Then
\[
p_u=\inf\{p\in[0,1]: \inf_{x\in V}\mathbf{P}_p(o\leftrightarrow x)>0  \}.
\]

Schonmann then conjectured this is also true for nonunimodular case (Conjecture 3.1 in \cite{S2001}).
Following Schonmann \cite{S2001} we denote the right-hand side of the above equality as follows:
\begin{definition}
	$$\overline{p}_{\textnormal{conn}}:=\inf\{p\in[0,1]: \inf_{x\in V}\mathbf{P}_p(o\leftrightarrow x)>0  \}.$$
\end{definition}
Note the right hand side does not depend on the choice of $o\in V$ by Harris-FKG inequality. Since $\mathbb{P}_p(o\leftrightarrow x)$ is a nondecreasing function of $p$, we also have 
\[
\overline{p}_{\textnormal{conn}}=\sup\{p\in[0,1]: \inf_{x\in V}\mathbf{P}_p(o\leftrightarrow x)=0  \}.
\]
An application of the above Theorem \ref{thm:heavy clusters are indis} is the following extension of Theorem 4.1 of \cite{LS1999a} in the Bernoulli percolation setting for all quasi-transitive graphs (we use Theorem 1.1 to deal with the nonunimodular case while the unimodular case is already proved in Theorem 4.1 of \cite{LS1999a}), and this also  confirms Schonmann's Conjecture 3.1 in \cite{S2001}.
\begin{theorem}\label{thm:connectivity decay}
	Suppose $G=(V,E)$ is a quasi-transitive graph and $\mathbf{P}_p$ is a Bernoulli bond percolation on $G$. If $\mathbf{P}_p$ has more than one infinite cluster a.s., then connectivity decays, that is
	\[
	\inf\{\ \mathbf{P}_p(x\leftrightarrow y): x,y\in V \}=0.
	\]
	In particular, 
	\[
	p_u(G)=\overline{p}_{\textnormal{conn}}.
	\]
\end{theorem}
\begin{remark}
	The above theorems also hold for  Bernoulli site percolation by similar arguments.  
\end{remark}

The proofs of the above two theorems follow similar strategies as the ones of unimodular transitive case. The main new ingredient is that we introduce certain ``biased"  random walks to replace the role of simple random walk in the study of properties of Bernoulli percolation clusters on nonunimodular quasi-transitive graphs. Mass transport principle in its general form also plays a key role in the proof of the stationarity of the ``biased" random walks. \\

Grimmett, Kesten and Zhang \cite{GKZ1993} first proved that the infinite cluster for supercritical Bernoulli percolation on  $\mathbb{Z}^d$ with $d\geq 3$ is transient for simple random walk. In \cite{LS1999a} Lyons and Schramm showed that if $G$ is a locally finite, connected infinite  graph with a transitive unimodular closed automorphism group $\Gamma\subset\textnormal{Aut}(G)$, and $(\mathbf{P},\omega)$ is a $\Gamma$-invariant insertion-tolerant percolation process on $G$ that has almost surely infinitely many infinite clusters, then a.s. each infinite cluster is transient (Proposition 3.11 of \cite{LS1999a}). Benjamini, Lyons and Schramm conjectured that if $G$ is a transient Cayley graph, then a.s. every infinite cluster of Bernoulli percolation on $G$ is transient (Conjecture 1.7 in \cite{BLS1999}). One may even conjecture that the same conclusion is true for all transient quasi-transitive graphs. Here we give a positive answer for the nonamenable quasi-transitive case. The case for general amenable quasi-transitive graphs remains open. 

\begin{theorem}\label{thm:transience of supercritical infinite clusters}
	Suppose $G$ is a quasi-transitive nonamenable graph. Then  a.s. every infinite cluster in Bernoulli percolation is transient for simple random walk.
\end{theorem}

Benjamini, Lyons and Schramm  proved a stronger result (Theorem 1.3 in \cite{BLS1999}) for nonamenable Cayley graphs, namely for Bernoulli percolation on such graphs, simple random walk on the infinite clusters has positive speed a.s. Their result can be easily generalized to hold for every quasi-transitive unimodular nonamenable graph. So to prove Theorem \ref{thm:transience of supercritical infinite clusters}, it suffices to deal with the nonunimodular case; see Proposition \ref{prop: transience for bernoulli percolation clusters}.

The paper is organized as follows: in Section 2 we introduce some preliminary results and notations. In Section 3 we review the tilted mass transport principle introduced in \cite{Tom2017}. Some properties of heavy clusters are discussed in Section 4, where their proofs also illustrate the applications of the tilted mass transport principle. We introduce the ``biased"  random walks  in Section 5 and the stationary property for them. We also prove Theorem  \ref{thm:transience of supercritical infinite clusters} in Section 5.  We prove Theorem \ref{thm:heavy clusters are indis} and Theorem \ref{thm:connectivity decay} in Section 6. Finally in Section 7 we discuss some examples of nonunimodular transitive graphs and also give some further questions.

\section{Preliminary}

Let $G=(V,E)$ be a locally finite, connected infinite unoriented graph with vertex set $V$ and edge set $E$. If $e=(u,v)\in E$, we say $u,v$ are adjacent and denote by $u\sim v$. An automorphism of $G$ is a bijection from $V=V(G)$ to itself which preserves adjacency. Let $\textnormal{Aut}(G)$ denote the group of all automorphisms of $G$. With the topology of pointwise convergence, $\textnormal{Aut}(G)$ is a locally compact Hausdorff topological group. Suppose $\Gamma\subset \textnormal{Aut}(G)$ is a closed subgroup with the induced topology. For any $v\in V$, the orbit of $v$ under $\Gamma$ is denoted by $\Gamma v:=\{\gamma v:\gamma\in\Gamma\}$. We use $G/\Gamma:=\{\Gamma v:v\in V\}$ to denote the orbit sets and say that $\Gamma$ is \textbf{quasi-transitive} or $\Gamma$ acts on $G$ quasi-transitively if there are only finite many orbits, namely $G/\Gamma$ is a finite set. In particular, if $G/\Gamma$ has a single element, then we say $\Gamma$ is \textbf{transitive} or $\Gamma$ acts on $G$ transitively. 
The graph $G$ is called quasi-transitive (transitive) if $\textnormal{Aut}(G)$ acts on $G$ quasi-transitively (transitively).

An infinite set of vertices $V_0\subset V$ is called end convergent if for any finite set $K\subset V$, there is a connected component of $G\backslash K$ that contains all but finitely many vertices of $V_0$. And two end convergent sets $V_0,V_1$ are said to be equivalent if $V_0\cup V_1$ is also end convergent. Now an end of $G$ is defined to be an equivalence class of end-convergent sets. For example, suppose $G$ is an infinite tree. Fix some $\rho\in V$ and call it the root of $G$. Then each end corresponds bijectively to a ray starting from the root $\rho$. The number of ends of $G$ is equal to the supremum of the number of components of $G\backslash K$ over all finite sets $K\subset V$. For example $\mathbb{Z}$ has two ends and $\mathbb{Z}^d$ has one end for $d\geq2$.   

Recall that on every locally compact Hausdorff topological group $\Gamma$, there is a unique (up to a multiplicative constant) Borel measure $|\cdot|$ that is invariant under left multiplication, namely $|A|=|\gamma A|$ for every Borel set $A\subset \Gamma$ and all $\gamma\in \Gamma$. This measure $m$ is called the (left) Haar measure. If it is also invariant under right multiplication, then $\Gamma$ is called a \textbf{unimodular group}. 
A quasi-transitive graph $G$ is said to be unimodular if its automorphism group is unimodular.

For each $x\in V$, $\Gamma_x:=\{\gamma\in\Gamma:\gamma x=x\}$ is called the stabilizer of $x$.
Let $m(x)=|\Gamma_x|$ denote the Haar measure of the stabilizer of $x$. And $\Gamma_xy:=\{\gamma y:\gamma\in\Gamma_x\}$ is the orbit of $y$ under $\Gamma_x$.

Examples of unimodular graphs include transitive amenable graphs \cite{SW1990} and Cayley graphs. Grandparent graphs \cite{Trofimov1985} and Diestel-Leader graphs \cite{Woess2005} $DL(k,n)$ with $k\neq n$ are typical examples of nonunimodular transitive graphs. More examples of nonunimodular transitive graphs can be found in Tim\'{a}r's paper \cite{Timar2006}. 

Even if a quasi-transitive graph $G$ itself is unimodular, there might exist a subgroup $\Gamma\subset  \textnormal{Aut}(G)$ that is nonunimodular and acts on $G$ quasi-transitively. Based on \cite{SW1990} such graphs must be nonamenable. A well-known example is the regular tree $T_3$ with degree $3$.
The regular tree $T_3$ is a unimodular transitive graph. Fix an end $\xi$ of $T_3$. Let $\Gamma$ be the subgroup of all automorphisms that fix $\xi$. Then $\Gamma$ is nonunimodular and acts on $T_3$ transitively. Other  examples will be discussed in more detail in the last section.

Suppose $\Gamma\subset \textnormal{Aut}(G)$ acts on $G$ quasi-transitively. A simple criterion for unimodularity of $\Gamma$ is provided by the following Proposition.
\begin{proposition}[Trofimov \cite{Trofimov1985}]\label{prop:unicre}
	$\Gamma$ is unimodular if and only if for all $x,y$ in the same orbit under $\Gamma$,  
	\[
	|\Gamma_xy|=|\Gamma_yx|,
	\]
	where $|\Gamma_xy|$ denotes the number of elements in the set $\Gamma_xy$.
\end{proposition}
From this proposition one can easily see that in the above example $T_3$, the subgroup $\Gamma$ fixing a specific end $\xi$ is nonunimodular. 

A well-known result concerning the Haar measure is the following lemma (for a proof, see for example formula (1.28) and Lemma 1.29 in \cite{Woess2000} ):
\begin{lemma}\label{lem:haar}
	Suppose $\Gamma\subset \mycmdA$ acts on $G=(V,E)$ quasi-transitively, then for all $x,y\in V$
	\[
	\frac{m(x)}{m(y)}=\frac{|\Gamma_xy|}{|\Gamma_yx|}=\frac{m(\gamma x)}{m(\gamma y)},\forall \gamma\in \Gamma.
	\]
	
\end{lemma}
A well-known criterion for unimodularity is the following  simple application of Lemma \ref{lem:haar}.
\begin{lemma}\label{lem:unicre}
	Suppose $\Gamma\subset \mycmdA$ acts on $G=(V,E)$ quasi-transitively and $m$ is an associated Haar measure, then $\Gamma$ is unimodular if and only if $\{m(\Gamma_x): x\in V\}$ is a finite set.
\end{lemma}
\begin{proof}
	If $\Gamma$ is unimodular, then by Proposition \ref{prop:unicre} and Lemma \ref{lem:haar} $m(\Gamma_x)$ is a constant on each orbit. Since there are only finitely many orbits,  $\{m(\Gamma_x), x\in V\}$ is finite.
	
	If $\Gamma$ is nonunimodular, then by Proposition \ref{prop:unicre} and Lemma \ref{lem:haar} there exist $x\in V,\gamma\in\Gamma$ such that  $M:=\frac{m(\gamma x)}{m(x)}>1$. Since by Lemma \ref{lem:haar} one has $\frac{m(\gamma^n x)}{m(\gamma^{n-1}x)}=\frac{m(\gamma^{n-1}x)}{m(\gamma^{n-2}x)}=\cdots=\frac{m(\gamma x)}{m(x)}=M>1$, $\frac{m(\gamma^n x)}{m(x)}=M^n\rightarrow\infty$ as $n\rightarrow\infty$, in particular $\{m(\Gamma_x), x\in V\}$ is not a finite set.
\end{proof}

Now we recall some terminology from percolation theory. Suppose $G=(V,E)$ is a locally finite, connected graph. Let $2^E$ be the collection of all subsets $\eta\subset E$. Let $\mathcal{F}_E$ be the $\sigma$-field generated by  sets of the form $\{\eta:e\in\eta\}$ where $e\in E$.  A bond percolation on $G$ is a pair $(\mathbf{P},\omega)$, where $\omega$ is a random element in $2^E$ and $\mathbf{P}$ is the law of $\omega$. For simplicity sometimes we will just say $\omega$ is a bond percolation. One can also define site percolation and mixed percolation on $G$ and the interested reader can refer to \cite{LP2016} or \cite{Grimmett1999} for more background on percolation. If $\omega$ is a bond percolation, then $\widehat{\omega}:=V\cup \omega$ is a mixed percolation on $2^{V\cup E}$. We will think of $\omega$ as a subgraph of $G$ and not bother to distinguish $\omega$ and $\widehat{\omega}$.

We will mainly focus on Bernoulli bond percolation on $G$, which we now define. The Bernoulli($p$) bond percolation on $G$ is a pair $(\mathbf{P}_p,\omega)$ that satisfies $\mathbf{P}_p$ is a product measure on $2^E$ and $\mathbf{P}_p[e\in\omega]=p$ for all edge $e\in E$. Sometimes we also write Bernoulli($p$) percolation as Bernoulli $p$-percolation.

If $v\in V$ and $\omega$ is a bond percolation on $G$, the cluster $C_\omega(v)$ (or just $C(v)$) of $\omega$ is defined as the connected component of $v$ in $\omega$. We sometimes also use $C(v)$ to denote the set of vertices (or edges) of the connected component of $v$ in $\omega$. 

If $(\mathbf{P},\omega)$ is a bond percolation on $G$ and $\Gamma$ is a subgroup of $\textnormal{Aut}(G)$, we call $(\mathbf{P},\omega)$  $\Gamma$-invariant if $\mathbf{P}$ is invariant under each $\gamma\in\Gamma$. Bernoulli percolation is obviously $\textnormal{Aut}(G)$-invariant, in particular Bernoulli percolation is $\Gamma$-invariant for any subgroup $\Gamma\subset \textnormal{Aut}(G)$. 

Next we recall some definitions for indistinguishability of infinite clusters. The following definition of invariant property is adapted from Definition 1.8 of \cite{TomAsaf2017}.
\begin{definition}\label{def: invariant property}
	Suppose $G=(V,E)$ is locally finite, connected infinite graph and $\Gamma$ is a closed quasi-transitive subgroup of $\textnormal{Aut}(G)$. For a measurable set $\mathcal{A}$ of $V\times\{0,1\}^E$, we say $\mathcal{A}$ is a \textbf{$\Gamma$-invariant} property, if $\gamma\mathcal{A}=\mathcal{A}$ for every $\gamma\in\Gamma$ and 
	\[
	(v,\omega)\in\mathcal{A}\ \Rightarrow\  \ \forall u\in C_{\omega}(v),\ (u,\omega)\in \mathcal{A}.
	\] We say that a cluster $C$ of $\omega$ has property $\mathcal{A}$ (and abuse notation by writing $C\in\mathcal{A}$) if $(u,\omega)\in \mathcal{A}$ for some (and hence every) vertex $u\in C$. 
\end{definition}

For example, $\mathcal{A}$ might be the collection  of all connected heavy subgraphs  of $G$, or the collection of  all recurrent subgraphs of $G$. For other examples see \cite{LS1999a,TomAsaf2017}. 

\begin{definition}[Definition 3.1 of \cite{LS1999a}]\label{def: indistinguishable infinite clusters}
	Suppose $G=(V,E)$ is locally finite, connected infinite graph and $\Gamma$ is a closed quasi-transitive subgroup of $\textnormal{Aut}(G)$.  Let 
	$(\mathbf{P},\omega)$ be a $\Gamma$-invariant bond percolation process on $G$. We say that $\mathbf{P}$ has ($\Gamma$-) \textit{indistinguishable infinite clusters} if for every $\Gamma$-invariant property $\mathcal{A}$, almost surely, for all infinite clusters $C$ of $\omega$, the cluster $C$  has property $\mathcal{A}$, or none of the infinite clusters has property $\mathcal{A}$. 
\end{definition}

\begin{definition}[Definition 3.2 of \cite{LS1999a}]\label{def: insertion-deletion}
	Let $G=(V,E)$ be a graph. Given a configuration $\omega\in 2^E$ and an edge $e\in E$, denote $\Pi_e\omega:=\omega\cup\{e\}$. For a  set $\mathcal{A}\subset 2^E$, we write $\Pi_e\mathcal{A}:=\{\Pi_e\omega: \ \omega\in\mathcal{A}\}$. A bond percolation process $(\mathbf{P},\omega)$ on $G$ is said to be \textit{insertion tolerant} if $\mathbf{P}[\Pi_e\mathcal{A}]>0$ for every $e\in E$ and every measurable $\mathcal{A}\subset 2^E$ with $\mathbf{P}[\mathcal{A}]>0$. A percolation $\omega$ is \textit{deletion tolerant} if $\mathbf{P}[\Pi_{\neg e}\mathcal{A}]>0$ for every $e\in E$ and every measurable $\mathcal{A}\subset 2^E$ with $\mathbf{P}[\mathcal{A}]>0$, where $\Pi_{\neg e}\omega:=\omega-\{e\}$. 
\end{definition}

For example Bernoulli($p$) percolation is insertion and deletion tolerant when $p\in(0,1)$.

For unimodular transitive graphs, Lyons and Schramm showed the following:
\begin{theorem}[Theorem 3.3 of \cite{LS1999a}]\label{thm:indisuni}
	Let $G$ be a graph with a transitive unimodular closed automorphism group $\Gamma\subset \textnormal{Aut}(G)$. Every $\Gamma$-invariant, insertion-tolerant, bond percolation process on $G$ has indistinguishable infinite clusters. 
\end{theorem}

Martineau's paper \cite{Martineau2015}  explored the link between ergodicity and indistinguishability in percolation theory. 

Example 3.15 of \cite{LS1999a} showed that a deletion-tolerant bond percolation could have distinguishable infinite clusters. Other examples of indistinguishability for non-insertion tolerant percolation processes see \cite{TomAsaf2017} and \cite{Timar2015}.

We also need a result from H\"{a}ggstr\"{o}m, Peres and Schonmann \cite{HPS1999}. 
\begin{definition}[Robust invariant property]
	Let $G=(V,E)$ be a graph and $\Gamma$ be a closed quasi-transitive subgroup of $\textnormal{Aut}(G)$. Let 
	$(\mathbf{P},\omega)$ be a $\Gamma$-invariant bond percolation process on $G$. Suppose $\mathcal{A}$ is a $\Gamma$-invariant property. We say that  $\mathcal{A}$ is a \textit{robust invariant property} if for every infinite connected subgraph $C$ of $G$ and every edge $e\in C$ we have the equivalence: $C$ has property $\mathcal{A}$
	iff there is an infinite connected component of $C\backslash\{e\}$ that has property $\mathcal{A}$. 
\end{definition}

Even without the unimodularity condition, the robust invariant properties do not distinguish between the infinite clusters of Bernoulli$(p)$ percolation  \cite{HPS1999}.
\begin{theorem}[Theorem 4.1.6 of \cite{HPS1999}]\label{thm:robust}
	Let $G$ be a quasi-transitive graph and   $p\in(p_c,p_u]$. If $Q$ is a robust invariant property such that  $\mathbf{P}_p(A)>0$, where $A$ is the event that there exists
	an infinite cluster satisfying $Q$, then $\mathbf{P}_p$-a.s., all infinite clusters in $G$ satisfy $Q$. 
\end{theorem}
An immediate application of Theorem \ref{thm:robust} is that almost surely light and heavy infinite clusters cannot coexist  for Bernoulli percolation.

\section{The Tilted Mass Transport Principle}

The mass transport principle turns out to be a very useful technique especially on unimodular transitive graphs \cite{BLPS1999}. Hutchcroft \cite{Tom2017} introduced a tilted version  of mass transport principle and it turns out be quite useful when applied to nonunimodular quasi-transitive graphs. 

Following Hutchcroft's notation  \cite{Tom2017} we define modular function as below:
\begin{definition}
	Let $G=(V,E)$ be a locally finite connected infinite graph and $\Gamma\subset \text{Aut}(G)$ acts on $G$ quasi-transitively. Let $\mathcal{O}=\{o_1,\ldots,o_L\}$ be a complete set of representatives in $V$ of the orbits of $\Gamma$. Let $a=(a_1,\ldots,a_L)$ be a sequence of positive real numbers with $\sum_{i=1}^{L}a_i=1$.  The \textbf{modular function} $\Delta_{\Gamma,a}:V^2\rightarrow(0,\infty)$ is defined as follows:
	\[
	\Delta_{\Gamma,a}(x,y)=\frac{a_ym(y)}{a_xm(x)}\stackrel{\textnormal{Lemma }\ref{lem:haar}}{=}\frac{a_y|\Gamma_yx|}{a_x|\Gamma_xy|},
	\]
	where $a_x:=a_i$ if  $x\in\Gamma o_i$. 
\end{definition}

Now we give the tilted mass transport principle (TMTP) (a slight generalization of Proposition 2.2 in Hutchcroft \cite{Tom2017}):
\begin{proposition}[TMTP]\label{prop:TMTP}
	With the same notations as in the above definition of modular function, suppose $f:V^2\rightarrow[0,\infty]$ is invariant under the diagonal action of $\Gamma$. Let $\rho$ be sampled from $\mathcal{O}$ with distribution $a$.  Then 
	\[
	\mathbb{E}[\sum_{x\in V}{f(\rho,x)}]=\mathbb{E}[\sum_{x\in V}{f(x,\rho)\Delta_{\Gamma,a}(\rho,x)}].
	\]
\end{proposition}
\begin{proof}
	This is easily seen from the Corollary 3.7 of \cite{BLPS1999}. 
\end{proof}

Hutchcroft made a particular choice for $a$ and proved for this particular choice that $\Delta_{\Gamma,a}$ has certain nice properties (Lemma 2.3 in \cite{Tom2017}). To be precise, write $[v]$ for the orbit of $v$ under $\Gamma$ and identify $\mathcal{O}$ with the  space of orbits. Let $(X_n)_{n\geq0}$ be a lazy simple random walk on $G$, namely $\mathbb{P}[X_{n+1}=u|X_n=v]=\frac{1}{2}\mathbf{1}_{\{u=v\}}+\frac{1}{2\textnormal{deg}(v)}\mathbf{1}_{\{u\sim v\}}$. Let $([X_n])_{n\geq0}$ be the corresponding Markov chain induced on the finite state space $\mathcal{O}$, which was called the \textbf{lazy orbit chain}. Note this chain is irreducible and hence it has a unique stationary measure $\widetilde{\mu}$ on $\mathcal{O}$. Let $\mu$ be the probability measure gotten from $\widetilde{\mu}$ biased by $\textnormal{deg}([v])^{-1}$:
\[
\mu([o])=\frac{\widetilde{\mu}([o])\textnormal{deg}([o])^{-1}}{\sum_{x\in\mathcal{O}}{\widetilde{\mu}([x])\textnormal{deg}([x])^{-1}}},\forall o\in\mathcal{O}.
\]
We also write $\mu_x:=\mu([x]),\forall x\in V$.
Then Lemma 2.3 in \cite{Tom2017} can be slightly strengthened to 
\begin{lemma}\label{lem:harmonicity of modular fcn}
	With the same notation as in the above definition of modular function, let $a=(a_1,\ldots,a_L)$ be a  sequence of positive numbers with $\sum_{i}^{L}a_i=1$. Then the modular function $\Delta=\Delta_{\Gamma,a}$ satisfies the following properties:
	\begin{enumerate}
		\item $\Delta$ is $\Gamma$ diagonally invariant.
		
		\item $\Delta$ satisfies the cocycle identity, i.e. for all $x,y,z\in V$ 
		\[
		\Delta(x,y)\Delta(y,z)=\Delta(x,z).
		\]
		
		\item Furthermore $a=\mu$, i.e. $a_i=\mu([o_i]),i=1,\ldots,L$ if and only if for some fixed $x\in V$ (hence for every $x\in V$), $\Delta(x,y)$ is a harmonic  function of $y\in V$, namely
		\[
		\Delta(x,y)=\frac{1}{\textnormal{deg}(y)}\sum_{z\sim y}{\Delta(x,z)},
		\]
		where the sum on the right hand side is taken with multiplicity if there are multiple edges between $y$ and $z$. 
	\end{enumerate}
\end{lemma}

\begin{proof}
	
	(1) For any $x,y\in V,\gamma\in\Gamma$ one has $a_{\gamma x}=a_x,a_{\gamma y}=a_y$. Hence
	\[
	\Delta(\gamma x,\gamma y)=\frac{a_{\gamma y}m(\gamma y)}{a_{\gamma x}m(\gamma x)}=\frac{a_ym(y)}{a_xm(x)}=\Delta(x,y).
	\]
	
	(2) For any $x,y,z\in V$ one has 
	\[
	\Delta(x,y)\Delta(y,z)=\frac{a_ym(y)}{a_xm(x)}\frac{a_zm(z)}{a_ym(y)}=\frac{a_zm(z)}{a_xm(x)}=\Delta(x,z).
	\]

	(3) The direction that $a=\mu$ implies the harmonicity of $\Delta(x,\cdot)$ was already shown in Lemma 2.3 in \cite{Tom2017} and here we provide an alternative proof. If one defines the transition matrix $\widetilde{P}$ on the factor chain on $\mathcal{O}$ 
	by $\widetilde{p}(o_i,o_j)=\frac{1}{\textnormal{deg}(o_i)}\sum_{x\in [o_j]}\mathbf{1}_{\{o_i\sim x\}}$, then $\widetilde{\mu}$ is also a stationary probability measure for the Markov chain on $\mathcal{O}$ determined by $\widetilde{p}$. By Lemma 3.25 of \cite{Woess2000} one has that $\nu(x):=\widetilde{\mu}([x])|\Gamma_x|=\widetilde{\mu}([x])m(x)$ is a stationary measure for simple random walk on $G$. From this it follows that for $a=\mu$ the function $\Delta(x,\cdot)$ is harmonic for any fixed $x\in V$.

	On the other hand, if for some fixed $x$, the function $y\mapsto \Delta(x,y)$ is harmonic, then  one has the following system of equations:
	\begin{equation}\label{eq:harmonicity}
		\left\{
		\begin{array}{ccc}
			0&=&\sum_{z\sim y}[a_zm(z)-a_ym(y)],y\in\mathcal{O}\\
			1&=&a_1+\ldots+a_L
		\end{array}
		\right.
	\end{equation}
	We already showed that when $a=\mu$, the function $y\mapsto \Delta(x,y)$ is harmonic, whence $a=\mu$ satisfies the above equations (\ref{eq:harmonicity}) and then it suffices to show $a=\mu$ is also the unique solution of (\ref{eq:harmonicity}). 
	
	Rewrite (\ref{eq:harmonicity}) as the following:
	\begin{equation}\label{eq:h2}
		\left\{
		\begin{array}{ccc}
			0&=&\sum_{z\sim y}[\frac{a_z}{\mu_z}\mu_zm(z)-\frac{a_y}{\mu_y}\mu_ym(y)],y\in\mathcal{O}\\
			1&=&a_1+\ldots+a_L
		\end{array}
		\right.
	\end{equation}
	Since $a=\mu$ satisfies (\ref{eq:harmonicity}), one can define a transition matrix $p_\mu$ on $\mathcal{O}$ by $p_\mu(y,z)=\frac{1}{\text{deg}(y)}\sum_{w\sim y,w\in[z]}{\frac{\mu_wm(w)}{\mu_ym(y)}}$. Since $G$ is connected, the Markov chain on the finite space $\mathcal{O}$ determined by $p_\mu$ is irreducible. Define $f(x):=\frac{a_x}{\mu_x},x\in V$. Then equation (\ref{eq:h2}) implies that $f$ is a harmonic function for $p_\mu$. 
	Since the Markov chain determined by $p_\mu$ is an irreducible Markov chain on a finite state space, the harmonic function $f$ must be constant. Together with $a_1+\ldots+a_L=\mu([o_1])+\ldots+\mu([o_L])=1$, we have $a=\mu$.

	This completes the proof. 
\end{proof}

\begin{remark}\label{rem:uni}
	If $\Gamma$ is unimodular, then on each orbit $m(\cdot)$ is a constant and then one obvious solution of (\ref{eq:harmonicity}) is $a_i=\frac{m(o_i)^{-1}}{\sum_{x\in\mathcal{O}}{m(x)^{-1}}}$. By the uniqueness of solution, one has $\mu([v])=\frac{m(v)^{-1}}{\sum_{x\in\mathcal{O}}{m(x)^{-1}}},\forall v\in V$. 
	Therefore in the case $a=\mu$, $\Delta_{\Gamma,\mu}(x,y)=\frac{\mu_ym(y)}{\mu_xm(x)}$ is a constant function of $y$ for fixed $x$, whence $\Delta_{\Gamma,\mu}(x,\cdot)$ is not only harmonic for simple random walk but also for any random walk  on $G$ associated with a $\Gamma$-invariant conductance. 
	
	Moreover by Theorem 3.1 of \cite{AL2007}, if one takes a random root $\rho$ from $\mathcal{O}$ with probability $\mathbb{P}(\rho=o_i)=\mu([o_i])$, then $[G,\rho]$ is a unimodular random rooted graph. 
	And by Theorem 4.1 of \cite{AL2007} with $p$ corresponding to simple random walk on $G$ and $\nu$ corresponding to degree, one has $\widehat{\mu}$ is a stationary probability measure for the trajectories on rooted graphs induced from simple random walk. In particular, the marginal of $\widehat{\mu}$ on the root of the initial rooted graph is a stationary measure for the lazy orbit chain which coincides with previous choice of $\widetilde{\mu}$. 
	
\end{remark}

Another natural question to ask is whether for a given $a=(a_1,\ldots,a_L)$ with $a_i>0,\sum_{i=1}^{L}a_i=1$, there is a deterministic $\Gamma$-invariant conductance function $c:E\rightarrow(0,\infty)$ such that the function $\Delta_{\Gamma,a}(x,\cdot)$ is harmonic for the network $(G,c)$. Previous harmonic function (correponding to simple random walk without explicit mentioned conductance) is just the case $c\equiv1$. 

For unimodular $\Gamma$ we have a complete answer.
\begin{lemma}\label{lem:uni-conductance}
	Suppose $\Gamma$ is unimodular and $a=(a_i)_{i=1}^L$ is a sequence of positive numbers with $\sum_{i=1}^{L}a_i=1$. Then $a=\mu$ if and only if there exists a $\Gamma$-invariant conductance $c:E\rightarrow(0,\infty)$ such that  $\Delta(x,\cdot)=\Delta_{\Gamma,a}(x,\cdot)$ is harmonic for network $(G,c)$. Also in the case $a=\mu$, $c$ can be chosen to be any $\Gamma$-invariant conductance.
\end{lemma}
\begin{proof}
	The if part is already given in the above Remark \ref{rem:uni}. 
	
	Now suppose a $\Gamma$-invariant function $c:E\rightarrow(0,\infty)$ is such that  $\Delta(x,\cdot)=\Delta_{\Gamma,a}(x,\cdot)$ is harmonic for network $(G,c)$. 
	Hence $a$ satisfies the following equations:
	\begin{equation}\label{eq:hamcon}
		\left\{
		\begin{array}{ccc}
			0&=&\sum_{z\sim y}[{a_zm(z)c(z,y)-a_ym(y)c(z,y)}],y\in\mathcal{O}\\
			1&=&a_1+\ldots+a_L
		\end{array}
		\right.
	\end{equation}
	where $c(z,y)=c(y,z)$ denote the conductance of the unoriented edge $e=(y,z)$.

	Consider the Markov chain on $\mathcal{O}$  with transition probability $q_c$ given by
	\[
	q_c(o_i,o_j)=\frac{\sum_{z\sim o_i,z\in[o_j]}{c(o_i,z)}}{\sum_{x\sim o_i}{c(o_i,x)}},\forall o_i,o_j\in\mathcal{O}.
	\]
	Since $\Gamma$ is unimodular, $m(\cdot)$ is a constant function on each orbit, and then we can define $f:\mathcal{O}\rightarrow(0,\infty)$  by $f([z]):=a_zm(z)$. Then by (\ref{eq:hamcon}) $f$ is a harmonic function for the  Markov chain determined by $q_c$. Since $G$ is connected, this Markov chain is irreducible. Moreover this Markov chain has finite state space $\mathcal{O}$, thus $f$ must be a constant function. Thus one has $a_z\propto m(z)^{-1}$. From  Remark \ref{rem:uni} one has $a=\mu$. 
\end{proof}	

For the case $\Gamma$ is nonunimodular, we do not have a complete answer except the transitive case.

For example, let $T$ be the infinite regular tree with degree $3$ and let $G$ be the graph obtained from $T$ by adding a new vertex at the midpoint of each edge of $T$. Fix an end $\xi$ of $T$ and let $\Gamma$ be the subgroup fixing this end $\xi$. Then $\Gamma$ is nonunimodular and quasi-transitive on $G$. Fix $x\in T$ and let $x_1\in T$ denote the neighbor of $x$ such that $x_1$ is  closer to $\xi$. And let $x_2,x_3$ be the other two neighbors of $x$ in $T$. Suppose $y_i$ is the midpoint on edge $(x,x_i)$ for $i=1,2,3$. Then $\mathcal{O}=\{x,y_1 \}$ is a complete set of representatives. Write $a_1=a_x$ , $a_2=a_{y_1}$ and let $\lambda:=\frac{a_1}{a_2}$. Since $a_1+a_2=1$,  $a_1=\frac{\lambda}{1+\lambda},a_2=\frac{1}{1+\lambda}$.  Solving \eqref{eq:hamcon} yields that for any $\lambda\in (\frac{1}{2},1)$, there exists a unique (up to multiplicative constant) $\Gamma$-invariant conductance such that $\Delta_{\Gamma,a}(x,\cdot)$ is harmonic w.r.t. this conductance. For other $\lambda$, the corresponding function $\Delta_{\Gamma,a}(x,\cdot)$ is not harmonic w.r.t. any $\Gamma$-invariant conductance.

Note in the transitive case $L=1,a_1=1=\mu$ and thus the choice of $a$ does not matter. And the function $\Delta_{\Gamma,a}(x,\cdot)=\frac{m(\cdot)}{m(x)}$ in this case. 
\begin{lemma}\label{lem:nonuni-conductance-transitive}
	Suppose $\Gamma$ is nonunimodular and transitive. Then $\Delta(x,\cdot)=\frac{m(\cdot)}{m(x)}$ is harmonic for any  $\Gamma$-invariant conductance $c:E\rightarrow(0,\infty)$.
\end{lemma}
\begin{proof}
	Let  $c:E\rightarrow(0,\infty)$ be a $\Gamma$-invariant conductance. It suffices to show for some fixed $y\in V$, one has
	\[
	\sum_{z\sim y}c(z,y)m(z)=\sum_{z\sim y}{c(z,y)m(y)}.
	\]
	Define $f:V^2\rightarrow[0,\infty]$ by $f(u,v)=c(u,v)\mathbf{1}_{\{u\sim v\}}$. Then $f$ is invariant under $\Gamma$-diagonal action since $c$ is $\Gamma$-invariant. Suppose $\mathcal{O}=\{y\}$ and then the random vertex $\rho$ with law $\mu$ is just $y$.   By TMTP one has 
	\[
	\sum_{z\in V}{f(\rho,z)}=\sum_{z\in V}{f(z,\rho)}\Delta(\rho,z).
	\]
	That is just $\sum_{z\sim y}c(y,z)=\sum_{z\sim y}c(y,z)\frac{m(z)}{m(y)}$ and we are done.
\end{proof}

\section{ Some Properties of Heavy Percolation Clusters}

In this section we review some important properties for heavy percolation clusters. Some of their proofs  are also typical applications of TMTP. We begin with a definition of level set.

\begin{definition}
	Suppose $G$ is a connected, locally finite graph, and $\Gamma\subset\textnormal{Aut}(G)$ is closed and quasi-transitive. A \notion{level}
	 of $(G,\Gamma)$ is a maximal set $X$ of vertices such that for any $x,y\in X$, $|\Gamma_xy|=|\Gamma_yx|$. Or equivalently, a level is a maximal set $X$ of vertices such that for any   $x,y\in X$, $m(x)=m(y)$. 
\end{definition}
Note the level set depends on the subgroup $\Gamma$. In the following without explicitly mention we will always fix some subgroup $\Gamma$. 

The next proposition was implicitly used in the proof of Corollary 5.10 in \cite{Timar2006} and we provide its proof for completeness.
\begin{proposition}\label{prop: level set is infinite}
	Suppose $G$ is a connected, locally finite graph, and $\Gamma\subset\textnormal{Aut}(G)$ is closed, quasi-transitive and nonunimodular, then every level is an infinite set.
\end{proposition}
\begin{proof}
	Let $L_x:=\{y\in V: m(y)=m(x)\}$ denote the level set containing $x$. 
	By  Lemma \ref{lem:haar} one has
	\begin{equation}\label{eq:bij}
		L_{\gamma x}=\gamma L_x:=\{\gamma y:y\in L_x\},\ \forall \gamma\in \Gamma
	\end{equation}
	Choose a complete set of representatives $\mathcal{O}=\{o_1,\ldots,o_L\}$ for the orbit under $\Gamma$. Denote $L_i:=L_{o_i}$,  let $n_i=|L_i|$ denote the cardinality of $L_i$ and define $N=\max\{n_1,\ldots,n_L \}$.
	
	First we show if there exists one index $i\in\{1,\ldots,L\}$ such that $n_i=\infty$, then $n_j=\infty$ for every $ j\in\{ 1,\ldots,L \}$. 
	Without loss of generality we assume $n_1=\infty$, note $L_1=\bigcup_{j=1}^L L_1\cap \Gamma o_j$, whence there exists at least one $j\in\{ 1,\ldots,L \}$ such that $|L_1\cap \Gamma o_j|=\infty$. Thus there exists an infinite set $\{\gamma_k:k\in\mathbb{N}\}\subset\Gamma$ such that $m(o_1)=m(\gamma_k o_j), \forall k\in\mathbb{N}$ and $\{\gamma_k o_j:k\in\mathbb{N} \}$ is an infinite set. Then  using Lemma \ref{lem:haar} again for every $i\in\{1,\ldots,L\}$, one has that  $m(\gamma_0 o_i)=m(\gamma_1 o_i)=m(\gamma_2 o_i)=\ldots$. Since $\gamma_k$ preserves the graph distance,  $\{\gamma_k o_i:k\in\mathbb{N} \}$ is also an infinite set. By the fact that  $L_{\gamma_0 o_i}\supset\{\gamma_k o_i:k\in\mathbb{N} \}$ one has that $L_{\gamma o_i}$ is also an infinite set, whence by \eqref{eq:bij} $n_i=\infty$ for every $i\in\{ 1,\ldots,L\}$.  
	
	Now by the above result and \eqref{eq:bij} it suffices to show $N=\infty$. For any $M>0$, since $\Gamma$ is nonunimodular, from the proof of  Lemma \ref{lem:unicre} one has that there exists some $x_0,y_0\in V$ such that $M_0:=\frac{m(y_0)}{m(x_0)}>LM$.  Define $f:V^2\rightarrow[0,\infty]$ as 
	$$f(x,y)=\mathbf{1}_{\{d(x,y)=d(x_0,y_0),\frac{m(y)}{m(x)}=M_0\}}.$$ 
	Obviously $f$ is $\Gamma$-diagonally invariant.   Take  $a=(a_1,\ldots,a_L)$ such that $a_i=\frac{1}{L},\sum_{i=1}^{L}a_i=1$, then $\Delta_{\Gamma,a}(x,y)=\frac{m(y)}{m(x)}$. 
	Let $k_y$ denote  the number of vertices $x$ such that $d(x,y)=d(x_0,y_0)$ and $\frac{m(y)}{m(x)}=M_0$, and $k:=\max\{k_{o_i}: o_i\in \mathcal{O}\}$. 
	Applying the tilted mass transport principle one has 
	\[
	\frac{1}{L}\leq \sum_{i=1}^{L}\frac{1}{L}\sum_{x\in V}{f(o_i,x)}=\sum_{i=1}^{L}\frac{1}{L}\sum_{x\in V}{f(x,o_i)}\frac{1}{M_0}\leq \frac{k}{M_0}.
	\]
	Thus $k\geq \frac{M_0}{L}>M$. Suppose $k=k_{o_i}$, then this means there exists at least $k>M$ vertices $x$ such that $\frac{m(o_i)}{m(x)}=M_0$, in particular all these $k$ vertices belong to the same level set, whence $N\geq k>M$. Since this is true for arbitrary $M$, one has $N=\infty$. 
\end{proof}

The following proposition was mentioned and used in \cite{Timar2006} without a proof, we provide its proof for the reader's convenience and later use. For an automorphism $\gamma$ and a level $A$, if $\gamma A=A$, then we say $\gamma$ fixes the level $A$.
\begin{proposition}\label{prop:quasi-transitive for finite union of levels}
	Suppose $G$ is a connected, locally finite graph, and $\Gamma\subset\textnormal{Aut}(G)$ is closed, quasi-transitive and nonunimodular, let $\Gamma'\subset \Gamma $ be the subgroup that fixes a level (hence fixes every level). Suppose $H$ is a connected component of some finite union of levels and $\Gamma_H'$ is the subgroup of $\Gamma'$ that preserves $H$. Then $\Gamma_H'$  acts quasi-transitively on $H$, and it is closed and unimodular. 
\end{proposition}
\begin{proof}
	\textit{Step 1}:  First we show that for any $\gamma\in \Gamma$, if there exists $x\in V(G)$ such that $m(x)=m(\gamma x)$, then $\gamma A=A$ for every level $A$. 
	 Actually this is immediate from Lemma \ref{lem:haar}:
	\[
	m(x)=m(\gamma x)\textnormal{ and }\frac{m(x)}{m(y)}=\frac{m(\gamma x)}{m(\gamma y)}\Rightarrow m(y)=m(\gamma y).
	\]
	In particular one has $\Gamma_x\subset \Gamma',\forall x\in V$. 
	
	\textit{Step 2}:
	Suppose $L_{x_i},i=1,\ldots,n$ are $n$ distinct levels and $G_n:=\cup_{i=1}^{n}{L_{x_i}}$ be the union of these $n$ levels, viewed as the subgraph induced by the vertices in these levels. $G_n$ might be disconnected, however it has finitely many types of connected components up to isomorphisms. Indeed, for any $y_1\neq y_2\in L_{x_i}\cap \Gamma o_j$, there exists some $\gamma_1,\gamma_2\in\Gamma$ such that $y_1=\gamma_1o_j,y_2=\gamma_2o_j$ and $m(x)=m(y_1)=m(y_2)$. Thus $\gamma_2\gamma_1^{-1}y_1=y_2$. Since $m(y_1)=m(y_2)$, $\gamma_2\gamma_1^{-1}\in\Gamma'$ and is an isomorphism from the connected component of $y_1$ to the one of $y_2$. In particular there are at most $Ln$ types of connected components up to isomorphism, where $L$ is the number of orbits for the action of $\Gamma$ on $G$.
	
	\textit{Step 3}:
	For any connected component $H\subset G_n$, let $\Gamma'_H$ be the subgroup of $\Gamma'$ consisting of those maps which fix this component $H$. In particular, $\Gamma'_x=\Gamma_x\subset \Gamma'_H,\forall x\in H$. We now show that $\Gamma'_H$ acts on $H$ quasi-transitively and it is closed and unimodular. Notice $m$ is also a nonzero Haar measure restricted on $\Gamma'_H$ and $(\Gamma'_H)_x=\Gamma_x$. Also $m(\Gamma_x)$ is bounded on $H$ since $G_n$ has only finitely many levels. Therefore $\Gamma'_H$ is unimodular by Lemma \ref{lem:unicre}. 
	
	It is easy to see that $\Gamma'$ is closed. Indeed for any $\gamma\notin\Gamma'$, there exists $x\in V$ such that $m(\gamma x)\neq m(x)$. Therefore as an open neighborhood of $\gamma$, the set $\{\beta:\beta(x)=\gamma(x)\}$ is in the complement of $\Gamma'$, whence $\Gamma'$ is closed. Similarly one can show that $\Gamma'_H$ is also closed. 
	
	For any $y_1\neq y_2\in H\cap L_{x_i}\cap \Gamma o_j$, as in step 2 there exists some $\gamma\in\Gamma'$ such that $\gamma y_1=y_2$. Since $\gamma$ preserves graph distance, it maps the connected component of $y_1$ to the connected component of $y_2$, i.e. $\gamma$ maps $H$ to $H$, whence $\gamma\in\Gamma'_H$. Therefore under $\Gamma'_H$ there are at most $Ln$ different orbits for $H$ (here $L$ denotes the number of orbits for $\Gamma$ acting on $G$), whence 
	$\Gamma'_H$  acts quasi-transitively on $H$.
\end{proof}

The following proposition is well-known (\cite{Timar2006},\cite{Tom2017}) and its proof will be omitted. We stress that one needs only the assumption of $\Gamma$-invariance. 
\begin{proposition}\label{prop:unique cluster is heavy}
	Suppose $G$ is a connected, locally finite graph, and $\Gamma\subset\textnormal{Aut}(G)$ is closed, quasi-transitive and nonunimodular, let $(\mathbf{P},\omega)$ be a $\Gamma$-invariant percolation on $G$. If $\mathbf{P}$-a.s.\ there is a unique infinite cluster, then the unique infinite cluster must be heavy. 
\end{proposition}

\begin{definition}
	H\"{a}ggstr\"{o}m, Peres and Schonmann \cite{HPS1999} introduced the heaviness transition:
	\[
	p_h(G,\Gamma):=\inf\{p\in[0,1]: \mathbf{P}_p\textnormal{-a.s. there exists a heavy cluster }   \}.
	\]
	In the case $\Gamma=\textnormal{Aut}(G)$, one denote $p_h=p_h(G)=p_h(G,\textnormal{Aut}(G))$. 
\end{definition}
Note that if $\Gamma$ is closed, quasi-transitive and unimodular, $p_h(G,\Gamma)=p_c$. Using the canonical coupling one can see that for all $p>p_h$, $\mathbf{P}_p$-a.s. there exists a heavy cluster. By Theorem \ref{thm:robust} a.s.\ all infinite clusters are heavy for every $p>p_h(G,\Gamma)$.
An immediate consequence of  Proposition \ref{prop:unique cluster is heavy} is that $p_h(G,\Gamma)\leq p_u$. Hutchcroft \cite{Tom2017} proved that $p_c(G)<p_h(G,\Gamma)$ if $\Gamma\subset \textnormal{Aut}(G)$ is closed, quasi-transitive and nonunimodular.

The next proposition is important for later use. It is proved for the transitive case and the proof can be easily adapted to quasi-transitive case. 
\begin{proposition}[Corollary 5.6 of \cite{Timar2006}]\label{prop:finite levels}
	Suppose $G$ is a connected, locally finite graph, and $\Gamma\subset\textnormal{Aut}(G)$ is quasi-transitive and nonunimodular. Let $(\mathbf{P}_p,\omega)$ be a Bernoulli$(p)$ percolation on $G$. If $\mathbf{P}_p$-a.s. there are infinitely many heavy clusters, then there exists some finite union of levels $L_{x_i},i=1,\ldots,n$ such that there exists a connected component $H_n$ of $G_n:=\cup_{i=1}^{n}L_{x_i}$ on which Bernoulli$(p)$ percolation induces infinitely many infinite open components. 
\end{proposition}

Suppose $\Gamma\subset \textnormal{Aut}(G)$ is quasi-transitive and nonunimodular. $m$ is a Haar measure on $\Gamma$. Let $N(p)$ denote the number of infinite clusters for Bernoulli ($p$) percolation on $G$. It is well known that $N(p)=0,1$ or $\infty$ a.s., see for example Theorem 7.5 of \cite{LP2016}. By Hutchcroft \cite{Tom2017}, one has $p_c<p_h(G,\Gamma)$. Thus $N(p_h(G,\Gamma))=1$ or $\infty$. We now clarify whether these clusters are heavy or light. 
\begin{proposition}\label{prop:light at p_h}
	Suppose $\Gamma\subset \textnormal{Aut}(G)$ is quasi-transitive and nonunimodular. If $N(p_h(G,\Gamma))=\infty$ a.s., then all these infinite clusters are light a.s. If $N(p_h(G,\Gamma))=1$ a.s., then the unique infinite cluster is heavy a.s. 
\end{proposition}
\begin{proof}
	If $N(p_h(G,\Gamma))=1$, then the unique infinite cluster is heavy a.s. by Proposition \ref{prop:unique cluster is heavy}.
	
	If $N(p_h(G,\Gamma))=\infty$, we proceed by contradiction. Suppose there is at least one heavy infinite cluster, then by Theorem \ref{thm:robust} one has that all these infinite clusters are heavy a.s. Then by Proposition \ref{prop:finite levels} there exist  finitely many levels $L_{x_i},i=1,\ldots,n$ such that $G_n:=\cup_{i=1}^{n}L_{x_i}$ induces infinitely many infinite open components for Bernoulli percolation  at $p_h(G,\Gamma)$. 
	
	Define \[
	p_c(G_n):=\min\{p_c(H):H\textnormal{ is a connented component of }G_n \},
	\]
	where if $H$ is a finite connected component we set $p_c(H)=1$. The minimum is achieved because  there are finitely many types of connected components up to isomorphisms by Proposition \ref{prop:quasi-transitive for finite union of levels}.
	
	Since there are infinitely many infinite open components for Bernoulli percolation  at $p_h(G,\Gamma)$ on $G_n$, by Proposition \ref{prop:finite levels} there exists a connected component $H_n$ such that Bernoulli $p_h(G,\Gamma)$-percolation on $H_n$ has infinitely many infinite open components. In particular, $p_c(H_n)\leq p_h(G,\Gamma)$.
	
	First consider the case $p_c(H_n)=p_h(G,\Gamma)$. By Proposition \ref{prop:quasi-transitive for finite union of levels} the subgroup $\Gamma'_{H_n}$  is quasi-transitive on $H_n$ and unimodular. Since Bernoulli $p_h(G,\Gamma)$-percolation on $H_n$ has infinitely many infinite open components, $H_n$ must be nonamenable. By \cite{BLPS1999b} or \cite{Tom2016} there are no infinite clusters at $p_c(H_n)$, which contradicts $p_c(H_n)=p_h(G,\Gamma)$ and Bernoulli $p_h(G,\Gamma)$-percolation on $H_n$ has infinitely many infinite open components.
	
	If $p_c(H_n)<p_h(G,\Gamma)$, take some $p\in(p_c(H_n),p_h(G,\Gamma))$, then Bernoulli $p$-percolation has at least one infinite cluster on $H_n$ a.s. This infinite cluster must be heavy since $\{m(x):x\in H_n\}$ is bounded. Hence Bernoulli $p$-percolation on $G$ would also have at least one heavy cluster a.s., which contradicts with the definition of $p_h(G,\Gamma)$.
	
	Therefore if $N(p_h(G,\Gamma))=\infty$ a.s., then all these infinite clusters are light a.s.
\end{proof}

\begin{corollary}\label{cor:p_h equal limit of p_c}
	Suppose $\Gamma\subset \textnormal{Aut}(G)$ is quasi-transitive and nonunimodular. Suppose that $(G_n)_{n\in\mathbb{N}}$ is an increasing \textit{exhausting} sequence of finite union of levels in the sense that for each level $L_x$ there exists $N_x>0$  such that $L_x\subset G_n$ whenever $n\geq N_x$. Then $\lim_{n\rightarrow\infty}p_c(G_n)$ exists and
	\begin{equation}\label{eq:limit of p_c is at least p_h}
		\lim_{n\rightarrow\infty}p_c(G_n)\geq p_h(G,\Gamma).
	\end{equation}
	
	Moreover the following are equivalent:\\
	$(1)$ There exists some $p\in(0,1)$ such that Bernoulli $(p)$ percolation on $G$ has infinitely many heavy clusters;\\
	$(2)$   $p_h(G,\Gamma)<p_u$;\\
	$(3)$   $\lim_{n\rightarrow\infty}p_c(G_n)< p_u$.\\
	If 	$\lim_{n\rightarrow\infty}p_c(G_n)\leq  p_u$ (slightly weaker than the above three conditions), then 
	\[
	\lim_{n\rightarrow\infty}p_c(G_n)= p_h(G,\Gamma).
	\]
\end{corollary}
\begin{proof}
	For any $p>p_c(G_n)$, Bernoulli $(p)$ percolation has at least one infinite cluster on $G_n$ a.s., which must be heavy. Hence Bernoulli $(p)$ percolation on $G$ would also have at least one heavy cluster a.s., whence $p\geq p_h(G,\Gamma)$. Therefore $p_c(G_n)\geq p_h(G,\Gamma)$ for all $n$. Note $p_c(G_n)$ is decreasing, thus the limit exists and satisfies \eqref{eq:limit of p_c is at least p_h}.
	
	$(1)\Rightarrow (2)$: Suppose $(1)$ holds but $(2)$ does not hold, then $p_h(G,\Gamma)=p_u$. For any $p<p_h(G,\Gamma)$ there is no heavy cluster a.s.\ by the definition of $p_h(G,\Gamma)$. At $p=p_h(G,\Gamma)$, if $N(p_h(G,\Gamma))=\infty$, then Proposition \ref{prop:light at p_h} yields all these infinite clusters are light a.s. If $N(p_h(G,\Gamma))=1$, then there exists only one heavy cluster a.s. For $p>p_h(G,\Gamma)=p_u$, there is a unique infinite cluster a.s.\ \cite{S1999}. Thus for all $p\in[0,1]$, there cannot exist infinitely many heavy clusters, contradicting  $(1)$. 
	
	$(2)\Rightarrow (3)$: For any $p\in(p_h(G,\Gamma),p_u)$, $N(p)=\infty$ since $p<p_u$. Since $p>p_h(G,\Gamma)$, almost surely all the infinite clusters are heavy by the definition of $p_h(G,\Gamma)$ and Theorem \ref{thm:robust}. 
	Then by Proposition \ref{prop:finite levels}
	there exists a finite union of levels $G_n$ such that some connected component $H_n$ of $G_n$ has the property that Bernoulli $(p)$ percolation on $H_n$ has infinitely many  infinite clusters. In particular $p\geq p_c(H_n)\geq p_c(G_n)\geq \lim_{n\rightarrow\infty}p_c(G_n)$.  Let $p\downarrow p_h(G,\Gamma)$ to obtain that 
	\begin{equation}\label{eq:4.3p_h geq limit}
		p_h(G,\Gamma)\geq \lim_{n\rightarrow\infty}p_c(G_n).
	\end{equation}
	Combining with $p_h(G,\Gamma)<p_u$ one has $(3)$. 
	
	$(3)\Rightarrow (1)$: By (\ref{eq:limit of p_c is at least p_h}) one has $p_h(G,\Gamma)<p_u$, whence for $p\in (p_h(G,\Gamma),p_u)$, Bernoulli $(p)$ percolation on $G$ has infinitely many heavy clusters.
	
	Now suppose	$\lim_{n\rightarrow\infty}p_c(G_n)\leq  p_u$. If $p_h(G,\Gamma)<p_u$, i.e. condition $(2)$ holds, then one has  (\ref{eq:4.3p_h geq limit}), this combining with (\ref{eq:limit of p_c is at least p_h})  yields $	\lim_{n\rightarrow\infty}p_c(G_n)= p_h(G,\Gamma)$. If $p_h(G,\Gamma)=p_u$ then by assumption $p_h(G,\Gamma)=p_u\geq \lim_{n\rightarrow\infty}p_c(G_n)$, this combining with (\ref{eq:limit of p_c is at least p_h}) also yields $	\lim_{n\rightarrow\infty}p_c(G_n)= p_h(G,\Gamma)$. 
\end{proof}

A necessary condition for $p_h(G,\Gamma)<p_u$ is provided by Tim\'{a}r \cite[Corollary 5.8]{Timar2006}, and Hutchcroft conjectured that it is also sufficient \cite[Conjecture 8.5]{Tom2017}. 
\begin{remark}
	The inequality in \eqref{eq:limit of p_c is at least p_h} can be strict. For example, consider the regular tree $\mathbb{T}_{d}$ with degree $d\geq 3$ and fix an end $\xi$ of $\mathbb{T}_{d}$. Let $\Gamma_\xi\subset \textnormal{Aut}(\mathbb{T}_{d})$ be the subgroup consisting of automorphisms that fixes this end $\xi$. Let $G=\mathbb{T}_{d}\times\mathbb{Z}$ and $\Gamma=\Gamma_\xi\times\textnormal{Aut}(\mathbb{Z})$. Then $\Gamma$ is transitive and nonunimodular. Notice $p_h(G,\Gamma)=p_u(G)<1$, where the first equality is due to Corollary 5.8 of \cite{Timar2006}. However, any union of finite consecutive  levels consists of infinitely many copies of a cartesian product of a finite tree and $\mathbb{Z}$, whence $\lim_{n\rightarrow\infty}p_c(G_n)=1> p_h(G,\Gamma)=p_u$ in this example. 
\end{remark}

\section{The ``biased" two-sided random walks}
We start with some notations from Lyons and Schramm  \cite{LS1999b}. 
Suppose $V$ is a countable infinite set and $\Gamma$ is a locally compact group acting on $V$ (on the left) and that all stabilizers of elements of $V$ have finite Haar measure. We also suppose that the quadruple $(\Xi,\mathcal{F},\mathbf{P},\Gamma)$ is  a measure-preserving dynamical system, namely $\Gamma$ acts measurably on the measure space $(\Xi,\mathcal{F},\mathbf{P})$ and  preserves the measure $\mathbf{P}$.

The space of trajectories of random walk is $V^{\mathbb{N}}$. Let $(\Xi,\mathcal{F})$ be a measurable space. Define the shift operator $\mathcal{S}:V^{\mathbb{N}}\rightarrow V^{\mathbb{N}}$ by
\[
(\mathcal{S}w)(n):=w(n+1),
\]
and let $\mathcal{S}$ act on $\Xi\times V^{\mathbb{N}}$ as 
\[
\mathcal{S}(\xi,w):=(\xi,\mathcal{S}w),\ \ \forall\  (\xi,w)\in \Xi\times V^{\mathbb{N}}.
\]
For $\gamma\in \Gamma$, we define its action on $\Xi\times V^{\mathbb{N}}$ by 
\[
\gamma(\xi,w):=(\gamma \xi,\gamma w),
\]
where $(\gamma w)(n):=\gamma(w(n))$.

Let $\mathcal{T}:V^{\mathbb{Z}}\rightarrow V^{\mathbb{Z}}$ be the natural extension of $\mathcal{S}$, namely,
\[
\mathcal{T}\widehat{w}(n)=\widehat{w}(n+1),\ \forall\ n\in\mathbb{Z},\widehat{w}\in V^{\mathbb{Z}},
\]
and as before let $\mathcal{T}$ act on $\{0,1\}^V\times V^{\mathbb{Z}}$ as $\mathcal{T}(\xi,\widehat{w})=(\xi,\mathcal{T}\widehat{w})$.

Define the projection maps $\pi:V^\mathbb{Z}\rightarrow V^\mathbb{N}$ as follows 
\[
\pi(\widehat{w})(n)=\widehat{w}(n),n\geq0,\forall \widehat{w}\in V^\mathbb{Z}, 
\]
and define $\pi^-:V^\mathbb{Z}\rightarrow V^\mathbb{N}$ as follows 
\[
\pi(\widehat{w})(n)=\widehat{w}(-n),n\geq0,\forall \widehat{w}\in V^\mathbb{Z}. 
\]

We call a measurable function $q:\Xi\times V\times V\rightarrow [0,1]$, written as $q:(\xi,x,y)\mapsto q_{\xi}(x,y)$, a random environment (from $\Xi$) if for all $\xi\in\Xi,x\in V$, we have $\sum_{y\in V}{q_{\xi}(x,y)}=1$. The natural action of $\Gamma$ on $q$ is given by $(\gamma q)(\xi,x,y):=q(\gamma^{-1}\xi,\gamma^{-1} x,\gamma^{-1} y)$. 
Given $x\in V$ and a measurable map $\xi\mapsto \nu_{\xi}(x)$ from $\Xi$ to $[0,\infty)$, let $\widehat{\mathbf{P}}_x$ denote the joint distribution on $\Xi\times V^{\mathbb{N}}$ of $\xi$ biased by $\nu_{\xi}(x)$ and the trajectory of the Markov chain determined by $q_\xi$ starting at $x$. Let $\mathcal{I}$ denote the $\sigma$-field of $\Gamma$-invariant events in $\Xi\times V^{\mathbb{N}}$. For examples one can refer to \cite{LS1999b} or see our application of the following general theorem:
\begin{theorem}[Theorem 3 of \cite{LS1999b}]\label{thm:key}
	Let $V$ be a countable set with a quasi-transitive action by a locally compact group $\Gamma$. Let $\{o_1,\ldots,o_L\}$ be a complete set of representatives of $\Gamma\backslash V$ and write $m_i=m(o_i)$. Let $(\Xi,\mathcal{F},\mathbf{P},\Gamma)$ be a measure-preserving dynamical system and $q$ be a $\Gamma$-invariant random environment from $\Xi$. Suppose that $\nu:(\xi,x)\mapsto \nu_{\xi}(x)$ is a $\Gamma$-invariant measurable mapping from $\Xi\times V\rightarrow [0,\infty)$ such that for each $\xi\in\Xi$,
	$x\mapsto m(x)\nu_{\xi}(x)$ is a stationary distribution for the Markov chain determined by $q_\xi$. Write
	\[
	\widehat{\mathbf{P}}:=\sum_{i=1}^{L}\widehat{\mathbf{P}}_{o_i}
	\]
	Then the restriction of $\widehat{\mathbf{P}}$ to the $\Gamma$-invariant $\sigma$-field $\mathcal{I}$ is an $\mathcal{S}$-invariant measure. If
	\[
	\sum_{i=1}^{L}\mathbf{E}[\nu_{\xi}(o_i)]=1
	\]
	then $\widehat{\mathbf{P}}$ is a probability measure. 
\end{theorem}

Now we give a natural extension of Theorem \ref{thm:key} for a two-sided random walk. 
\begin{theorem}\label{thm:key-extension}
	With the same notations as in Theorem \ref{thm:key}, let
	\[
	q_\xi^{\leftarrow}(x_1,x_2):=\frac{\nu_\xi(x_2)m(x_2)}{\nu_\xi(x_1)m(x_1)}q_\xi(x_2,x_1)
	\]
	Since $x\mapsto \nu_\xi(x)m(x)$ is a stationary measure for the Markov chain determined by $q_\xi$, $q_\xi^{\leftarrow}$ is  a transition probability. Moreover $q_\xi^{\leftarrow}$ is also $\Gamma$-invariant.
	
	For $(\xi,x)\in \Xi\times V$, let $\theta_\xi^x$ denote the law of a two-sided random walk $\widehat{w}\in V^{\mathbb{Z}}$ starting from $x$, namely $w=\pi\widehat{w}$ is a random walk starting from $x$ determined by $q_\xi$, and $w^-:=\pi^-(\widehat{w})$ is an independent random walk starting from $x$ determined by transition probability $q_\xi^{\leftarrow}$. Let $\Theta_x$ denote the joint law of $(\xi,\widehat{w})$ biased by $\nu_\xi(x)$. Write
	\[
	\mathbf{\Theta}:=\sum_{i=1}^{L}\Theta_{o_i}
	\]
	
	Let $\mathcal{I}_{\mathbb{Z}}$ denote the $\Gamma$-invariant $\sigma$-field. Then the restriction of $\mathbf{\Theta}$ to $\mathcal{I}_{\mathbb{Z}}$ is a $\mathcal{T}$-invariant measure. If $$\sum_{i=1}^{L}\mathbf{E}[\nu_\xi(o_i)]=1,$$ then $\mathbf{\Theta}$ is a probability measure. 
\end{theorem}
\begin{proof}
	This is just an adaptation of the proof of Theorem 3 of  \cite{LS1999b}. 
	Let $F$ be a nonnegative $\Gamma$-invariant measurable function on $\Xi\times V^{\mathbb{Z}}$. It suffices to show 
	\begin{equation*}
		\int_{\Xi\times V^{\mathbb{Z}}} F\circ \mathcal{T} d\mathbf{\Theta}  =\int_{\Xi\times V^{\mathbb{Z}}} F\,   d\mathbf{\Theta}.
	\end{equation*}
	
	A key observation is that 
	\[
	d\theta_\xi^x(\widehat{w})=\sum_{y\in V}{\frac{\nu_\xi(y)m(y)}{\nu_\xi(x)m(x)}}\mathbf{1}_{\{\mathcal{T}\widehat{w}(-1)=x \}}d\theta_\xi^y(\mathcal{T}\widehat{w}).
	\]	
	Define 
	\[
	f(x,y;\xi):=\nu_\xi(y)\frac{m(y)}{m(x)}\int_{V^\mathbb{Z}} \mathbf{1}_{\{\widehat{w}(-1)=x \}} F(\xi,\widehat{w})d\theta_\xi^y(\widehat{w}).
	\]
	Then we have 
	\begin{equation*}
		\int_{\Xi\times V^{\mathbb{Z}}} F\circ \mathcal{T} d\mathbf{\Theta}=
		\int_{\Xi\times V^{\mathbb{Z}}} F\, d\mathbf{\Theta} \circ \mathcal{T}^{-1}=
		\sum_{j=1}^{L}\sum_{y\in V}\int_{\Xi}  d\mathbf{P}(\xi) {f(o_j,y;\xi)}.
	\end{equation*}
	
	It's easy to verify that $f$ is $\Gamma$-invariant and hence so is $\mathbf{E}[f(x,y;\cdot)]$. Then mass transport principle (for example Lemma 1 of \cite{LS1999b}) yields that 
	\begin{eqnarray*}
		\int_{\Xi\times V^{\mathbb{Z}}} F\circ \mathcal{T} d\mathbf{\Theta}
		&=&	\sum_{j=1}^{L}\sum_{y\in V}\int_{\Xi}  d\mathbf{P}(\xi) \frac{m(y)}{m(o_j)}{f(y,o_j;\xi)}\\
		&=&\sum_{j=1}^{L}\int_{\Xi} d\mathbf{P}(\xi)\int_{V^\mathbb{Z}}\sum_{y\in V} \nu_\xi(o_j)\mathbf{1}_{\{y=\widehat{w}(-1) \}}F(\xi,\widehat{w})d\theta_\xi^{o_j}(\widehat{w})\\
		&=&\sum_{j=1}^{L}\int_{\Xi} d\mathbf{P}(\xi)\int_{V^\mathbb{Z}}\nu_\xi(o_j) F(\xi,\widehat{w})d\theta_\xi^{o_j}(\widehat{w})\\
		&=&\int_{\Xi\times V^{\mathbb{Z}}} F   d\mathbf{\Theta} ,
	\end{eqnarray*}
	which completes the proof.
\end{proof}
Now we summarize some properties of the backward random walk determined by $q_\xi^\leftarrow$. The proof is standard for reversed Markov chains and we omit it. 
\begin{proposition}\label{prop:backwardrw}
	With the same notation as in Theorem \ref{thm:key-extension}, one has the following properties:\\
	
	$(1)$ $x\mapsto \nu_\xi(x)m(x)$ is also a stationary distribution for the Markov chain determined by $q_\xi^\leftarrow$. And $(q_\xi^\leftarrow)^{\leftarrow}=q_{\xi}$.

	$(2)$ $q_\xi,q_\xi^\leftarrow$ induces same communicating classes. In particular, $q_\xi$ is irreducible iff $q_\xi^\leftarrow$ is  irreducible. Moreover on each communicating class $q_\xi$ is transient if and only if $q_\xi^\leftarrow$ is transient. 
	
\end{proposition}

Next we return to the Bernoulli percolation setting. First we consider simple random walk on a quasi-transitive graph $G=(V,E)$.
\begin{corollary}\label{cor:srw} 
	Let $G=(V,E)$ be a quasi-transitive graph with automorphism group $\Gamma=\textnormal{Aut}(G)$. 
	Let $\{o_1,\ldots,o_L\}$ be a complete set of representatives of $\Gamma\backslash V$. Suppose that $(\mathbf{P}_p,\xi)$ is Bernoulli$(p)$ bond percolation process on $G$, and let $\widetilde{\mathbf{P}}_{o}:=\mathbf{P}_p\times\mathbf{P}_{o}$, where $\mathbf{P}_o$ is the law of simple random walk on $G$ starting at the point $o$. Then there exist positive constants $c_i,i=1,\ldots,L$, summing to $1$ such that $\widehat{\mathbf{P}}:=\sum_{i=1}^{L}c_i\widetilde{\mathbf{P}}_{o_i}$ is a probability measure and the restriction of $\widehat{\mathbf{P}}$ to the $\Gamma$-invariant $\sigma$-field $\mathcal{I}$ is an $\mathcal{S}$-invariant measure. 
\end{corollary}
\begin{proof}
	In Theorem \ref{thm:key} we take  $(\Xi,\mathcal{F},\mathbf{P},\Gamma)=(\{0,1\}^E,\mathcal{F}_E,\mathbf{P}_p,\textnormal{Aut}(G))$. $\mathbf{E}$ will be the corresponding expectation operator with respect to this measure $\mathbf{P}=\mathbf{P}_p$.

	Take $q(\xi,x,y)=q_\xi(x,y)=\frac{\mathbf{1}_{\{x\sim y \}}}{\textnormal{deg}(x)}$, then the Markov chain associated with transition probability $q_\xi$ is just simple random walk on $G$ and obviously $q$ is $\Gamma$-invariant. 
	
	Fix some $u\in V$. We take $\nu:\{0,1\}^V\times V\rightarrow[0,\infty)$ to be
	\[
	\nu_{\xi}(x)=c\frac{\Delta(u,x)\textnormal{deg}(x)}{m(x)},
	\]
	where $c>0$ is a constant to be determined later and $a=\mu$ as in Lemma \ref{lem:harmonicity of modular fcn} such that $\Delta$ is harmonic for simple random walk.
	By Lemma \ref{lem:haar}
	\begin{eqnarray*}
		\nu_\xi(x)&=&c\frac{\widetilde{\mu}([x])\textnormal{deg}(u)|\Gamma_xu|}{\widetilde{\mu}([u])\textnormal{deg}(x)|\Gamma_ux|}\frac{\textnormal{deg}(x)}{m(x)}\\
		&=&c\frac{\textnormal{deg}(u)}{\widetilde{\mu}([u])m(u)}\widetilde{\mu}([x]).
	\end{eqnarray*}
	From this we see that $\nu$ is also $\Gamma$-invariant and is constant on each orbit. 
	
	From the harmonicity of modular function (Lemma \ref{lem:harmonicity of modular fcn}), we know that $x\mapsto m(x)\nu_\xi(x)$ is a stationary distribution for the Markov chain (simple random walk now) determined by $q_\xi$. We choose the normalizing constant $c>0$ to be determined by
	\[
	\sum_{i=1}^{L}\mathbf{E}[\nu_{\xi}(o_i)]=1.
	\]
	Now take $c_i=\nu_\xi(o_i)$ (this does not depend on $\xi$) for $i=1,\ldots,L$. Then $c_i\widetilde{\mathbf{P}}_{o_i}=\widehat{\mathbf{P}}_{o_i}$, where $\widehat{\mathbf{P}}_x$ is defined earlier as the joint distribution on $\Xi\times V^{\mathbb{N}}$ of $\xi$ biased by $\nu_{\xi}(x)$ and the trajectory of the Markov chain determined by $q_\xi$ starting at $x$.  Therefore Theorem \ref{thm:key} yields the desired result. 
\end{proof}

Next we consider the corresponding two-sided random walk as in Theorem \ref{thm:key-extension}. Since $q_\xi(x,y)=\frac{\mathbf{1}_{\{x\sim y \}}}{\textnormal{deg}(x)}$, one has 
\[
q_\xi^\leftarrow(x,y)=\frac{\nu_\xi(y)m(y)}{\nu_\xi(x)m(x)}q_\xi(y,x)=\frac{\mu([y])m(y)}{\mu([x])m(x)}
\frac{\mathbf{1}_{\{x\sim y \}}}{\textnormal{deg}(x)}.
\]
Note by Remark \ref{rem:uni} in the unimodular case $q_\xi^\leftarrow=q_\xi$ is just the transition probabilities for simple random walk. As in Theorem \ref{thm:key-extension} fix some $x\in V$ and let $w$ be a simple random walk on $G$ started from $x$, and $w^-$ be an independent random walk determined by $q_\xi^\leftarrow$ also started from $x$. Set $\theta^x$ to be the law of the two-sided random walk $\widehat{w}$ started from $x$ determined by $w=\pi \widehat{w}$ and $w^-=\pi^-\widehat{w}$. 

Let $\mathbf{P}_{\rho}^{\mathbb{Z}}:=\sum_{i=1}^{L}c_i\theta^{o_i}$ denote the law of a two-sided random walk defined above starting from the independently  random choosed vertex $\rho$, where $c_i$ is from Corollary \ref{cor:srw}.

With these notations and $\widehat{\mathbf{P}}$ as in Corollary \ref{cor:srw} , we have the following extension.
\begin{corollary}\label{cor:two-sided srw} 
	Let $\mathcal{I}_{\mathbb{Z}}$ denote the $\Gamma$-invariant $\sigma$ field of $\{0,1\}^E\times V^{\mathbb{Z}}$, then $(\{0,1\}^E\times V^{\mathbb{Z}},\mathcal{I}_{\mathbb{Z}},\mathbf{P}_p\times\mathbf{P}_{\rho}^{\mathbb{Z}},\mathcal{T})$ is an invertible measure preserving dynamical system. Moreover $\widehat{\mathbf{P}}=\mathbf{P}_p\times\mathbf{P}_{\rho}^{\mathbb{Z}}\circ\pi^{-1}$. 
\end{corollary}
\begin{proof}
	This is just an application of Theorem \ref{thm:key-extension} in this particular setting.
	
	It's straightforward to verify that $\widehat{\mathbf{P}}=\mathbf{P}_p\times\mathbf{P}_{\rho}^{\mathbb{Z}}\circ\pi^{-1}$ and we omit the details.
\end{proof}

Notice that in the above ``simple random walk on $G$" setting, $q_\xi$ actually is independent of the percolation configuration $\xi$. In the following we shall consider random walk on the percolation clusters of $\xi$.

We first recall the delayed two-sided simple random walk on percolation clusters from \cite{LS1999b}: let $(\mathbf{P},\xi)$ be a bond percolation process on $G$ and $\xi\in 2^E$ be a percolation configuration. Let $x\in V$ be some fixed vertex, called base point.  Let $w(0)=x$. For $n\geq0$, conditioned on $\langle w(0),\ldots,w(n)\rangle$ and $\xi$, let $w'(n+1)$ be chosen uniformly from the neighbors of $w(n)$ in $G$ with equal probability. If the edge $(w(n),w'(n+1))$ belongs to $\xi$, then we set $w(n+1)=w'(n+1)$, otherwise we set $w(n+1)=w(n)$. This $w$ is a delayed simple random walk on percolation cluster $C_\xi(x)$. Given $\xi$, let $w,w^-$ be two independent delayed simple random walk. Set $\widehat{w}$  such that $w=\pi (\widehat{w}),w^-=\pi^-(\widehat{w})$. Then $\widehat{w}$ is called a two-sided delayed simple random walk. If $G$ is transitive and unimodular, the two-sided delayed simple random walk $\widehat{w}$ is shift invariant on the $\Gamma$-invariant $\sigma$-field, which is a key ingredient in the proof of Theorem \ref{thm:indisuni}. 

Now for general quasi-transitive graphs, we introduce the following ``biased" two-sided random walk inspired by Example 6 in \cite{LS1999b}. Suppose $\Gamma\subset \textnormal{Aut}(G)$ is quasi-transitive and $m$ is an associated Haar measure on $\Gamma$. Let $(\mathbf{P},\xi)$ be a bond percolation process on $G$.  For each edge $e=(x,y)\in E$,  set conductance $c(e)=\sqrt{m(x)m(y)}$. Now given a  percolation configuration $\xi$, for each edge $e\in E$, if $e\notin \xi$, delete the edge $e$ from $G$ and add a loop at $x$ with conductance $c(e)$ and a loop at $y$ also  with conductance $c(e)$. In the resulting network, one has a corresponding random walk according the conductance, which is our desired ``biased" random walk.   To be specific, we define $q_\xi(x,y)$ as follows:
\begin{equation}\label{eq:biased transition probability}
	q(\xi,x,y)=q_\xi(x,y)=
	\left\{\begin{array}{cc}
		\frac{\sqrt{m(x)m(y)}}{\sum_{z\sim x}\sqrt{m(z)m(x)} }&\textnormal{ if }(x,y)\in\xi,y\neq x,\\
		& \\
		0&\textnormal{ if  }y\neq x \textnormal{ and } (x,y)\notin \xi,\\
		&\\
		1-\sum_{ z\sim x,z\neq x}{q_\xi(x,z)} &\textnormal{ if } y=x.
	\end{array}\right.
\end{equation}

Obviously $q$ and $\nu_\xi(x):=c\sum_{z\sim x}\sqrt{m(z)/m(x)}$ are both  $\Gamma$-invariant by Lemma \ref{lem:haar}, where $c>0$ is a constant. Moreover it is obvious that  $x\mapsto \nu_\xi(x)m(x)=c\sum_{z\sim x}\sqrt{m(z)m(x)}$ is a stationary measure for $q_\xi$. 
Let $q_\xi^\leftarrow$ be given as in Theorem \ref{thm:key-extension}, then
\[
q_\xi^\leftarrow(x,y)=\frac{\nu_\xi(y)m(y)}{\nu_\xi(x)m(x)}q_\xi(y,x)=q_\xi(x,y). 
\]
\begin{definition}\label{def:two-sided baised rw}
	Fix $x\in V$. Let $w$ be a ``square-root biased" random walk determined by $q_\xi$ in formula \eqref{eq:biased transition probability} started from $x$. Let $w^-$ be an independent ``square-root biased" random walk determined by $q_\xi^\leftarrow=q_\xi$ started from $x$. Let $\widehat{w}\in V^{\mathbb{Z}}$ be such that 
	$w=\pi \widehat{w},w^-=\pi^- \widehat{w}$. Then $\widehat{w}$ is called a two-sided ``square-root biased" random walk started from $x$. 
\end{definition}
Notice in the case $\Gamma$ is transitive and unimodular, $\widehat{w}$ is just the two-sided delayed simple random walk we recalled earlier.

\begin{corollary}\label{cor:two-sided biased rw}
	Suppose $G=(V,E)$ is a connected, locally finite graph and $\Gamma\subset \textnormal{Aut}(G)$ is closed and quasi-transitive. Let $m$ be a Haar measure on $\Gamma$.  
	Let $\{o_1,\ldots,o_L\}$ be a complete set of representatives of $\Gamma\backslash V$. Suppose that $(\mathbf{P},\xi)$ is a $\Gamma$-invariant bond percolation process on $G$, and let $q_\xi$ be given by formula \eqref{eq:biased transition probability}, $c>0$ be a constant, and $\nu_\xi(x)=c\sum_{z\sim x}\sqrt{m(z)/m(x)}$. For $\xi\in 2^E,x\in V$, let $\widehat{w}$ be a ``square-root biased" two-sided random walk  started from $x$ as  in Definition \ref{def:two-sided baised rw}.  Let $\Theta_x$ denote the joint law of $(\xi,\widehat{w})$ biased by $\nu_\xi(x)$. Write
	\[
	\mathbf{\Theta}:=\sum_{i=1}^{L}\Theta_{o_i}
	\]
	
	Let $\mathcal{I}_{\mathbb{Z}}$ denote the $\Gamma$-invariant $\sigma$-field. Then the restriction of $\mathbf{\Theta}$ to $\mathcal{I}_{\mathbb{Z}}$ is an $\mathcal{T}$-invaraint measure, where $\mathcal{T}$ is the natural extension on $\Xi\times V^\mathbb{Z}$ of $\mathcal{S}$. If  the constant $c>0$ satisfies $\sum_{i=1}^{L}\mathbf{E}[\nu_\xi(o_i)]=1$, then $\mathbf{\Theta}$ is a probability measure. 
\end{corollary}
\begin{proof}
	This is immediate from Theorem \ref{thm:key-extension}. 
\end{proof}

\begin{proposition}\label{prop:light cluster is recurrent}
	With the same notation as in Corollary \ref{cor:two-sided biased rw}, if $C$ is a light cluster of $\xi$, then the ``square-root biased" random walk determined by $q_\xi$ on $C$ is positive recurrent. 
\end{proposition}
\begin{proof}
	Since $\Gamma$ is quasi-transitive, there exists a constant $M>0$ such that for any $x\sim y$, $\frac{1}{M}\leq \frac{m(y)}{m(x)}\leq M$. Let $D$ be the maximal degree of $G$.
	Then consider the induced network on $C$ (described just before formula \eqref{eq:biased transition probability}), the total conductance is finite:
	\[
	\sum_{e\in C}{c(e)}\leq \sum_{x\in C}{D\sqrt{M}m(x)}<\infty,
	\]
	where in the last inequality we use the fact $C$ is light. 
	Therefore $C$ is positive recurrent w.r.t. $q_\xi$. 
\end{proof}

\begin{proposition}\label{prop:transience for heavy clusters}
	Suppose $\Gamma\subset \textnormal{Aut}(G)$ is quasi-transitive and nonunimodular, if Bernoulli $(p)$ percolation on $G$ has infinitely many heavy clusters a.s., then these heavy clusters are transient  for both simple random walk and the ``square-root biased" random walk determined by $q_\xi$ in the formula \eqref{eq:biased transition probability}. 
\end{proposition}
\begin{proof}
	By Proposition \ref{prop:finite levels} there exists  some finite union of levels $L_{x_i},i=1,\ldots,n$ such that there exists a connected component $H_n$ of $G_n:=\bigcup_{i=1}^{n}L_{x_i}$ with the property that  Bernoulli $(p)$ percolation on $H_n$ has infinitely many infinite clusters. Since $\Gamma'_{H_n}$ is quasi-transitive and unimodular, simple random walk is transient on infinite clusters of  Bernoulli $(p)$ percolation on $H_n$, see for example Proposition 3.11 of \cite{LS1999a} (its proof can be easily adapted to quasi-transitive case). Since conductance $c(e)=\sqrt{m(e^-)m(e^+)}$ is bounded on $H_n$, the random walk determined by $q_\xi$ is also transient on infinite clusters of  Bernoulli $p$-percolation on $H_n$. Therefore by Rayleigh's monotonicity principle, there exists some infinite cluster of Bernoulli $(p)$ percolation on $G$ such that simple random walk and the ``square-root biased" random walk determined by $q_\xi$ are both transient. Since transience for simple  random walk and the ``square-root biased" random walk  are both $\Gamma$-invariant robust properties, Theorem \ref{thm:robust}   yields the desired conclusion.
\end{proof}	
\begin{remark}\label{rem:remark after prop heavy transient}
	For general insertion-and-deletion tolerant percolation processes, by Corollary 5.6 and Remark 5.11 of \cite{Timar2006} we can show that with positive probability there is at least one heavy cluster that is transient for both simple random walk and the ``square-root biased" random walk.
\end{remark}

If $p_h(G,\Gamma)=p_u$, then there is no $p\in[0,1]$ such that there are infinitely many heavy clusters for Bernoulli$(p)$ percolation. We conjecture that in this case if  there is a unique infinite cluster for Bernoulli$(p)$ percolation a.s., then the unique infinite cluster is also transient for the ``square-root biased" random walk determined by $q_\xi$. Moreover the unique infinite cluster is transient for simple random walk; see Proposition \ref{prop: transience for bernoulli percolation clusters}.\\

In Remark 3.12 of \cite{LS1999a}, it was conjectured that if there are almost surely infinitely many infinite clusters for Bernoulli percolation, then almost surely every infinite cluster is transient (for simple random walk). Moreover there were examples of $\Gamma$-invariant insertion-tolerant percolation processes showing that infinite clusters can be recurrent for simple random walk in the case $\Gamma$ is nonunimodular. An explicit example by Russell Lyons is as follows: 
\begin{example}\label{example: insertion-tolerance}
	Let $T$ be a regular tree with degree $3$ and $\xi$ be a distinguished end of $T$. Let $\Gamma$ be the group of automorphisms that fixes the end $\xi$. Let every vertex of $T$ be connected to precisely one of its offspring (as measured from $\xi$), each with probability $1/2$. Let $\omega_1$ denote the configuration. Then every cluster of $\omega_1$ is a ray. Now for each edge $e=(x,y)\in E(T)\backslash\omega_1$ where $y$ is an offspring of $x$, let $n$ be the graph distance from $x$ to the highest vertex of its ray in $\omega_1$. Next we insert the edge $e$ with probability $\frac{1}{2^{n+1}}$. Let $\omega_2$ be the configuration gotten from $\omega_1$ by applying the above procedure independently for each edge $e=(x,y)\in E(T)\backslash\omega_1$. Then  $\omega_2$ is a $\Gamma$-invariant insertion-tolerant percolation process. It is easy to see that every cluster of $\omega_2$ is infinite and recurrent for simple random walk. 
\end{example}

We now show that the conjecture in Remark 3.12 of \cite{LS1999a} holds. Unlike the short proof of Proposition \ref{prop:transience for heavy clusters}, the proof of transience for all infinite Bernoulli percolation clusters  is much longer. Moreover this transience result is not needed for the proof of the main theorems.

First we briefly review  existing results. Suppose $G$ is a nonamenable transitive graph. For $p$ close to $1$, the anchored expansion constants of the infinite clusters of Bernoulli$(p)$ percolation are positive a.s.\ \cite{CP2004}, whence the liminf speeds of simple random walk on the infinite clusters are positive a.s.\ \cite{V2000}. If moreover $G$ is unimodular, then the speed of simple random walk on an infinite cluster exists and is positive \cite[Theorem 4.4]{BLS1999} for every Bernoulli$(p)$ percolation on $G$  with  $p>p_c$. In particular, this implies that the infinite clusters of  Bernoulli$(p)$ percolation on $G$  with  $p>p_c$ are transient. Proposition 3.11 of \cite{LS1999a} shows that the same is true for any $\textnormal{Aut}(G)$-invariant insertion-tolerant percolation process that has infinitely many infinite clusters a.s. 

Now we consider the case $G$ is nonunimodular and quasi-transitive. For $p>p_c$ and a delayed simple random walk on the infinite clusters of every Bernoulli$(p)$ percolation on $G$, we do not know whether the speed of the delayed simple random walk exists  since it is not shift invariant as in the unimodular case. However, the infinite clusters of  Bernoulli$(p)$ percolation on $G$  with  $p>p_c$ are still transient. 

\begin{proposition}\label{prop: transience for bernoulli percolation clusters}
	Suppose $G$ is nonunimodular and quasi-transitive and $p>p_c$. Then the infinite clusters of  Bernoulli$(p)$ percolation on $G$  are transient. 
\end{proposition}
The proof of Proposition \ref{prop: transience for bernoulli percolation clusters} is a simple modification of the proof of Theorem 4.3 of \cite{Timar2006}. 

First we point out that Lemma 4.1 of \cite{Timar2006} can be slightly strengthened to (the conditions are the same while the conclusion is slightly stronger):
\begin{lemma}\label{lem:transient for branching}
	Consider a random rooted tree with the following properties. Fix  some $p>0$ and define $O_0:=\{o\}$, where $o$ is the root. If a generation $O_m$ is already given, then the number of children that the vertices in $O_m$ will have depends only on $O_m$ (and not on the past). Each vertex of $O_m$ has at least $k$ children with probability $\geq p$ and $0$ children otherwise. Furthermore there is a positive integer $\alpha$ such that given any generation $O_m=\{v_1,\ldots,v_n\}$, if we let $X_i$ be the number of children of $v_i$, then for each $i\in\{1,\ldots,n\}$, $X_i$ is independent of $X_j$ for all but at most $\alpha$ of them.. Denote by $O_{m+1}$ the set of children of the vertices in $O_m$. 
	
	Then the tree is transient with positive probability whenever $kp>1$.
\end{lemma}
\begin{proof}
	We may assume that every vertex has exactly $k$ children with probability $p$ and $0$ children otherwise. Pick some $q\in(0,1)$ sufficiently close to $1$ such that $kpq^k>1$. Let $T$ be a random tree as stated in the Lemma. Consider a Bernoulli$(q)$ percolation on $T$ and construct a new random rooted tree $T'$ as follows: Let $V'$ be the set of vertices $x$ with the property that  it has exactly $k$ children $x_1,\ldots,x_k$ and every edge $(x,x_i)$ is open in the Bernoulli$(q)$ percolation on $T$. $T'$ is the subtree of $T$ induced by $V'\cup\{o\}$. Notice $T'$ is also a random rooted tree satisfies the property of Lemma 4.1 of \cite{Timar2006}, whence $T'$ is infinite with positive probability. This implies that $p_c(T)<1$ with positive probability, whence $\textnormal{br}(T)>1$ with positive probability (\cite[Theorem 5.15]{LP2016}) and in particular $T$ is transient with positive probability(\cite[Theorem 3.5]{LP2016}). 
\end{proof}	

We now adopt some notations from \cite{Tom2017,Timar2006}. Suppose $G=(V,E)$ is a locally finite, connected graph and  $\Gamma\subset \textnormal{Aut}(G)$ is nonunimodular and acts quasi-transitively on $G$. Let $\mathcal{O}=\{o_1,\ldots,o_L\}$ be a complete set of representatives of $G/\Gamma$. Let $a=\mu$ be as in Lemma \ref{lem:harmonicity of modular fcn} and then the modular function $\Delta(x,y):=\Delta_{\Gamma,a}(x,y)$ is harmonic for simple random walk. Let $\rho$ be a random vertex picked from $\mathcal{O}$ with ditribution $a$.

For each $s\leq t$ and $v\in V$ we define the slab
\[
S_{s,t}(v):=\{ x\in V: s\leq \log \Delta(v,x) \leq t    \}.
\]
We also define \[
t_0:=\sup\{\log\Delta(v,u):u,v\in V, u\sim v \}.
\]

We define the separating layers 
\[
L_n(v):=\{x\in V: (n-1)t_0\leq \log\Delta(v,x) \leq nt_0  \}
\]
and half spaces $H_n^+(v):=\bigcup_{m\geq n}L_m(v)$ and $H_n^-(v):=\bigcup_{m\leq n}L_m(v)$. We also define $L_{m,n}(v):=\bigcup_{k=m}^{n}L_k(v)$. 

For each $v\in V$, $-\infty\leq m\leq k\leq n\leq \infty$ we define 
\[
X_k^{m,n}(v):=|\{x\in L_k(v): v\stackrel{L_{m,n}(v)}{\longleftrightarrow}x \}|
\]
and 
\[
\widetilde{X}_k^{m,n}(v):=|\{x\in L_k(v): v\stackrel{L_{m,n}(v)}{\longleftrightarrow}x \textnormal{ by an open path with only } x\in L_k(v) \}|,
\]
where  $\{ v\stackrel{L_{m,n}(v)}{\longleftrightarrow}x\}$ denotes the event that $v$ is connected to $x$ by an open path in the subgraph $L_{m,n}(v)$.

We also need a modification of Lemma 4.2 from \cite{Timar2006} as follows.
\begin{lemma}\label{lem: intersection tend to infinity}
	Let $G$ be a nonunimodular quasi-transitive graph and $v$ be a vertex of $G$. Consider Bernoulli$(p)$ percolation on $G$ that has light infinite clusters a.s. Then given the event that $C(v)$ is an infinite light cluster, $\widetilde{X}_{-n}^{-n,\infty}(v)\rightarrow\infty$ as $n\rightarrow\infty$.
\end{lemma}
\begin{proof}
	The proof is almost the same as the one of Lemma 4.2 in \cite{Timar2006}. Just replace the event $E=E(k)$ there by the event that $C(v)$ is infinite and light and there are infinitely many $n$ such that $\widetilde{X}_{-n}^{-n,\infty}(v)\leq k$. 	
\end{proof}

\begin{proof}[Proof of Proposition \ref{prop: transience for bernoulli percolation clusters}]
	Denote by $\Gamma$ the group $\textnormal{Aut}(G)$ in this proof. Since $\Gamma$ is nonunimodular and acts quasi-transitively on $G$, one has $p_c<p_h$ by \cite{Tom2017}. Moreover there is no infinite cluster at $p_c$; see\cite{Tom2016}. For $p\in(p_c,p_h)$ Bernoulli$(p)$
	percolation has infinitely many infinite light clusters. Since for all $p_c<p_1<p_2\leq 1$, every infinite $p_2$-cluster contains an infinite $p_1$-cluster \cite[Theorem 4.1.3]{HPS1999}, by Rayleigh's monotonicity principle it suffices to show for any fixed $p\in (p_c,p_h)$ the light infinite clusters of Bernoulli$(p)$
	percolation  are transient for simple random walk. In the following we fix some $p\in (p_c,p_h)$ and some $o\in V$ such that there exists a neighbor $o'$ of $o$ such that $\log\Delta(o,o')=-t_0$. Recall $[o]:=\Gamma o$ denotes the orbit of $o$.

	First by quasi-transitivity and maximum principle for the harmonic function $\Delta(v,\cdot)$ there exists $n_0>0$ such that for each $v\in V$, there exists a path $P_v=(v_0=v,v_1,\ldots,v_n)$ such that $n\leq n_0$, $\log \Delta(v,v_n)\leq -t_0$, $ v_n\in [o]$ and $\Delta(v,v_i)$ is decreasing. In particular in the transitive case for $v=o$ we can choose $n_0=1$ and  $P_v=(o,o')$. 
	Indeed, we just need to choose one such path for each $x\in \mathcal{O}$. For every vertex $y\notin \mathcal{O}$, suppose $y$ is in the same orbit as $x$ under $\Gamma$, namely for some (arbitrarily fixed) $\gamma=\gamma_{x,y}\in\Gamma$,  $y=\gamma x$.  Then we let $P_y=\gamma P_x$ (the choice of $\gamma$ does not matter although it may change the choice of $P_y$).  Denote the end point $v_n$ of $P_v$ by $v'$.
	
	For $x\in \mathcal{O}$, define the graph $G'(x)$ to be the union of the path $P_x$ and the subgraph induced by the vertices in half space $H_{0}^-(x')$. For   $x\in \mathcal{O}$ and $y\in [x]\backslash\{x\}$ let $G'(y):=\gamma G'(x)$, where $\gamma=\gamma_{x,y}$ is the one fixed as in the above definition of $P_y$.
	
	We call the cluster $C(x)$  \textbf{nice} if $C(x)=C(x)\cap {G'(x)}$ and $C(x')\cap H_0^-(x')=C(x')\cap {G'(x')}$. Let $F(x)$ be the event that $C(x)$ is infinite, light and nice. By insertion and deletion tolerance and $p\in(p_c,p_h)$ we have $q_x:=\mathbf{P}_p(F(x))>0$. Moreover $q_x$  depends only on the orbit of $x$, whence $q:=\inf\{q_x:x\in V \}=\min\{q_x:x\in\mathcal{O} \}>0$. 
	
	Given a vertex $x\in L_{-j}(o)$ (if there are two such $j's$, we take the larger one), the vertex $x'$ and $r,i\geq1$, define $B_x(i;r):=B(x,r)\cap G'(x)\cap L_{-j-i,-j}(o)$. One can see Figure 1 on page 2354 of \cite{Timar2006} for an illustration of  $B_x(i;r)$ in the transitive case. We say a vertex $v\in B_x(i;r)\backslash P_x$ belongs to the \textbf{side boundary} of $B_x(i;r)$ if there is some edge $(v,w)$ such that $w\in H_{-j-i}^+(o)$ and $w\notin B_x(i;r)$. 
Let $C(x)|_{B_x(i;r)}$ denote the open cluster of $x$ for percolation restricted to the finite graph $B_x(i;r)$.  
	Let $Y_{-i}(o,x):=\{ v\in L_{-j-i}(o) : v\stackrel{L_{-i,0}(x)}{\longleftrightarrow}x \textnormal{ by an open path with only }v\in L_{-j-i}(o) \}$ and $\widetilde{X}_{-i}(o,x):=|Y_{-i}(o,x)|$.
	Given $k\geq 1$, we say that ${B_x(i;r)}$ is $k$-\textbf{good} if $\widetilde{X}_{-i}(o,x)\geq k$ and the side boundary of $B_x(i;r)$ is disjoint from $C(x)|_{B_x(i;r)}$.
	
	\textbf{Claim:} for any given $k$ there is a uniform choice of $i,r$ such that $B_x(i;r)$ is $k$-good with probability at least $q/2$ for every $x\in H_0^-(o)$. 
	
	Suppose $x\in L_{-j}(o)$. From Lemma \ref{lem: intersection tend to infinity}  for any given $\lambda>1$, there exists a positive integer $i_x=i_x(\lambda)$ such that given the event $F(x)$ occurs, with probability at least $\frac{3}{4}$, $\widetilde{X}_{-i+1}^{-i+1,\infty}(x)\geq \lambda k$ if $i= i_x$. Let $Z_i(x)$ denote the set of vertices that contribute to $\widetilde{X}_{-i+1}^{-i+1,\infty}(x)$. Now condition on  $\widetilde{X}_{-i+1}^{-i+1,\infty}(x)\geq \lambda k$, we have  $\widetilde{X}_{-i}(o,x)\geq k$ with probability at least $\frac{3}{4}$. Indeed, if $|Z_i(x)\cap L_{-j-i}(o)|\geq k$, then $\widetilde{X}_{-i}(o,x)\geq k$. Otherwise, $|Z_i(x)\cap L_{-j-i+1}(o)|\geq (\lambda-1)k$. For each $v\in Z_i(x)\cap L_{-j-i+1}(o)$, consider the part of the path $P_v$ stopped at its first hitting time of $L_{-j-i}(o)$ and call it $P_v(o,x)$. If $P_v(o,x)$ is also open, then the endpoint (other than $v$) belongs to $Y_{-i}(o,x)$. Notice the event $\widetilde{X}_{-i+1}^{-i+1,\infty}(x)\geq \lambda k$ is independent of $P_v(o,x)$ is open for every $v\in Z_i(x)\cap L_{-j-i+1}(o)$ and two such paths are disjoint if $d(v,v')\geq 2n_0$. Therefore if $\lambda$ is sufficiently large, with probability at least $\frac{3}{4}$ there are at least $k$ such paths are open. Thus $\widetilde{X}_{-i}(o,x)\geq k$ with probability at least $\frac{3}{4}$ conditioned on $\widetilde{X}_{-i+1}^{-i+1,\infty}(x)\geq \lambda k$.
	Combining the above we have $\widetilde{X}_{-i}(o,x)\geq k$ has probability at least $(\frac{3}{4})^2q_x\geq \frac{9q}{16}$. Now since $C(x)$ is light a.s., $C(x)$ intersect each slab $L_{-i}(x)$ with at most finitely many vertices a.s. Thus there exists a large integer $r=r_x$ such that probability of the event that $C(x)$ intersect the boundary of $B(x,r)$ at some vertex $v\in H_{-i}^+(x)$ is at most $\frac{q}{16}$. 
	
	Notice if  $F(x)$ occurs, $ \widetilde{X}_{-i}(o,x)\geq k $  and $C(x)$ does not intersect the boundary of $B(x,r)$ at some vertex $v\in H_{-i}^+(x)$, then $B_x(i;r)$ is $k$-good. Hence $B_x(i;r)$ is $k$-good with probability at least $\frac{q}{2}$. Notice $i,r$ depend only on the orbit of $x$ and this proves the \textbf{Claim} above. 
	
	Now we can proceed to construct a random tree as in the proof of Theorem 4.3 of \cite{Timar2006}. We fix $k$ such that $\frac{kq}{2D^{n_0}}>1$ and then some $i,r$ such that $B_x(i;r)$ is $k$-good with probability at least $q/2$ for every $x\in H_0^-(o)$. 
	The vertex $o\in V(G)$ corresponds to the root $\widehat{o}$  of $T$, the $0$ generation $O_0$ of $T$. If $B_o(i;r)$ is $k$-good, then it contains at least $k$ vertices in $L_{-i}(o)$ that can be connected to $o$ by an open path with only one endpoint lying in $L_i(o)$. For each of these vertices $x$, add a child $\widehat{x}$ to $\widehat{o}$ in $T$ and let these vertices $\widehat{x}$ constitute the first generation $O_1$ of $T$. If $B_o(i;r)$ is not $k$-good, let $\widehat{o}$ have $0$ children and $T=\{\widehat{o}\}$. 
	
	Suppose we have defined the $g$th generation $O_g$ of $T$ such that each vertex $\widehat{x}\in O_g$ corresponds to a vertex $x\in L_{-gi}(o)$. We can partition $O_g$ such that $\widehat{x},\widehat{y}$ are in the same class of the partition iff for the corresponding $x,y\in L_{-gi}(o)$, we have $x'=y'$.
	Each set of the partition has at most $D^{n_0}$ elements, where $D$ the maximum degree of a vertex in $G$.
	Now choose one vertex in each class of the partition uniformly and independently; call the set of chosen vertices \textit{parental} vertices. If $\widehat{x}$ is not parental, then let it have $0$ children.
	
	If $\widehat{x}$ is parental, assign $f$ children to it iff $B_x(i;r)$ is $k$-good and $f:=\widetilde{X}_{-i}(o,x)\geq k$; assign $0$ children to $\widehat{x}$ otherwise.  Assigning children in this way for each $\widehat{x}\in O_g$ and these children constitute $O_{g+1}$. Note that a vertex has at least $k$ children with probability at least $\frac{q}{2D^{n_0}}$. Notice different  vertices in $O_{g+1}$ will also correspond to different vertices in $L_{-(g+1)i}(o)$.
	
	It is straightforward to verify that the tree $T$ constructed above satisfies the condition of Lemma \ref{lem:transient for branching} and the interested reader can refer to the proof of Theorem 4.3 in \cite{Timar2006} for a similar verification. Hence $T$ is transient with positive probability. Notice the open cluster restricted in the subgraph induced by the corresponding vertices $x$ and $B_x(i,r)$
	is roughly isometric to $T$, whence it is also transient. By Rayleigh's monotonicity principle $C(o)$ is transient with positive probability. Since transience is a robust invariant property, Theorem \ref{thm:robust} yields the desired conclusion. 
\end{proof}

\section{Proofs of the Main Theorems}

The proof of Theorem \ref{thm:heavy clusters are indis} follows a similar strategy as the proof of Theorem \ref{thm:indisuni} in \cite{LS1999a}, and the ``square-root biased" two-sided random walk in Definition \ref{def:two-sided baised rw} will play the role of two-sided delayed simple random walk in the proof of Theorem \ref{thm:indisuni}.

\begin{definition}\label{def:pivotal edge}
	Suppose $\Gamma\subset\textnormal{Aut}(G)$ is quasi-transitive and $(\mathbf{P},\xi)$ is a $\Gamma$-invariant bond percolation process on $G$. Suppose $\mathcal{A}$ is a $\Gamma$-invariant property. An infinite cluster $C$ of $\xi$ is called of type $\mathcal{A}$ if $C\in\mathcal{A}$; otherwise $C$ is called of type $\neg\mathcal{A}$. 
	
	Suppose that there is an infinite cluster $C$ of $\xi$ and $e\in E\backslash C$ such that the connected component $C'$ of $\xi\cup\{e\}$ that contains $C$ has a type different from the one of $C$. Then $e$ is called a pivotal edge for $(C,\xi)$.
\end{definition}

The following lemma is proved for transitive graphs in \cite{LS1999a} and the proof can be easily adapted to quasi-transitive ones. 
\begin{lemma}[Lemma 3.5 of \cite{LS1999a}]\label{lem:pivotal edge}
	Suppose $\Gamma\subset\textnormal{Aut}(G)$ is quasi-transitive. $(\mathbf{P},\xi)$ is an insertion-tolerant  $\Gamma$-invariant bond percolation process on $G$.  Suppose $\mathcal{A}\in \mathcal{F}_E$ is a $\Gamma$-invariant property. Assume that there is positive probability for coexistence of infinite clusters of type $\mathcal{A}$ and $\neg\mathcal{A}$. Then with positive probability, there is an infinite cluster $C$ of $\xi$ that has a pivotal edge. 
\end{lemma}

\begin{proof}[Proof of Theorem \ref{thm:heavy clusters are indis}]
	We proceed by contradiction. Suppose there exists some $\Gamma$-invariant property $\mathcal{A}$ such that there is positive probability for coexistence of heavy clusters in $\mathcal{A}$ and $\neg\mathcal{A}$. 
	Let $\mathcal{O}=\{o_1,\ldots,o_L \}$ be a complete set of representatives of $G/\Gamma$. Fix a Haar measure $m$ on $\Gamma$.  By Lemma \ref{lem:pivotal edge} we may assume with  positive probability  there are pivotal edges of heavy clusters of type $\mathcal{A}$ since otherwise one can replace $\mathcal{A}$ by $\neg\mathcal{A}$. For every $x\in V$ fix some $r_x>0$ such that with positive probability, the cluster $C(x)$ is heavy, of type $\mathcal{A}$ and there is an edge $e$ at graph distance at most $r_x$ from $x$ that is pivotal for $C(x)$. Notice we can choose $r_x$ only depending on the orbit of $x$. Let $r:=\max\{r_x:x\in\mathcal{O} \}$.

	Fix $\varepsilon>0$. 
	Define $A_x$ to be the event that the cluster of $x$ in $\xi$
	is heavy and of type $\mathcal{A}$. Let $A_x'$ be an event that depends on only finitely many edges such that $\mathbf{P}_p[A_x \Delta A_x' ]<\varepsilon$. Let $R_x$ be large enough such that $A_x'$ only depends on edges in the ball $B(x,R_x)$. Let $R:=\max\{R_x:x\in\mathcal{O} \}$.

	Take a random root $\rho\in\mathcal{O}$ independent of the Bernoulli percolation $(\mathbf{P}_p,\xi)$ with distribution
	$\mathbf{P}(\rho=o_i)=\nu_\xi(o_i)$, where $\nu_\xi(x)=c\sum_{z\sim x}\sqrt{\frac{m(z)}{m(x)}}$ and $c$ is such that $\sum_{i=1}^{L}\mathbf{E}[\nu_\xi(o_i)]=1$. 
	
	Let $W$ be a ``square-root biased" two-sided random walk given by Definition \ref{def:two-sided baised rw} started from the random root $\rho$. For $n\in\mathbb{Z}$, let $e_n\in E$ be an edge chosen uniformly among the edges within distance  at most $r$ from $W(n)$. Recall in Corollary \ref{cor:two-sided biased rw} $\mathbf{\Theta}$ is the joint law of $(\xi,W)$. Write $\widehat{\mathbf{P}}$ for the joint law of $\rho$, $\xi$, $W$ and $\langle e_n:n\in\mathbb{Z}\rangle$. 
	
	Given $e\in E$, let $\mathscr{P}_e$ be the event that $\xi\in A_\rho$ and $e$ is pivotal for $C(\rho)$. Let $\mathscr{E}_e^n$ be the event that $e_n=e$ and $W(j)$ is not an endpoint of $e$ whenever $-\infty<j<n$. Recall the insertion operation $\Pi_e$ in Definition \ref{def: insertion-deletion}. For any measurable event $\mathscr{B}$, $e\in E$ and $n\geq1$ one has that
	\[
	\widehat{\mathbf{P}}[\mathscr{E}_e^n\cap (\mathscr{B}\backslash\Pi_e\mathscr{B})]=\frac{1-p}{p}\widehat{\mathbf{P}}[\mathscr{E}_e^n\cap \Pi_e(\mathscr{B}\backslash\Pi_e\mathscr{B})]\leq \frac{1-p}{p}\widehat{\mathbf{P}}[\mathscr{E}_e^n\cap \Pi_e\mathscr{B}],
	\]
	where the first equality uses the definition of $\mathscr{E}_e^n$ and the fact that $n\geq1$.

	Then for all measurable event $\mathscr{B}$ one has
	\begin{eqnarray*}
		\widehat{\mathbf{P}}[\mathscr{E}_e^n\cap\mathscr{B}]&=&
		\widehat{\mathbf{P}}[\mathscr{E}_e^n\cap (\mathscr{B}\backslash\Pi_e\mathscr{B})]+\widehat{\mathbf{P}}[\mathscr{E}_e^n\cap (\mathscr{B}\cap\Pi_e\mathscr{B})]\\
		&\leq&\frac{1-p}{p}\widehat{\mathbf{P}}[\mathscr{E}_e^n\cap \Pi_e\mathscr{B}]+\widehat{\mathbf{P}}[\mathscr{E}_e^n\cap \Pi_e\mathscr{B}]\\
		&=&\frac{1}{p}\widehat{\mathbf{P}}[\mathscr{E}_e^n\cap \Pi_e\mathscr{B}].
	\end{eqnarray*}

	Applying the above inequality with $\mathscr{B}=A_\rho'\cap \mathscr{P}_e$ one has 
	\begin{equation}\label{eq:3.4}
		\widehat{\mathbf{P}}[\mathscr{E}_e^n\cap \Pi_eA_\rho'\cap \Pi_e\mathscr{P}_e]
		\geq p\widehat{\mathbf{P}}[\mathscr{E}_e^n\cap A_\rho'\cap \mathscr{P}_e]=p\widehat{\mathbf{P}}[\mathscr{E}_e^n\cap A_\rho'\cap \mathscr{P}_{e_n}].
	\end{equation}
	Define $\mathscr{E}^n:=\cup_{e\in E}\mathscr{E}_e^n$, and $\mathscr{E}_R^n:=\cup_{e\in E\backslash B(\rho,R)}\mathscr{E}_e^n$ and note that these are disjoint unions.
	
	By definition of $\mathscr{P}_e$, $\Pi_e\mathscr{P}_e\subset \neg A_\rho$ since the insertion of pivotal edge $e$ would change the type of $C(\rho)$. Also $\Pi_e A_\rho'\subset  A_\rho'$ for any edge $e\in E\backslash B(\rho,R) $ since $A_\rho'$ only depends edges within graph distance $R$ from $\rho$. Thus
	\begin{eqnarray}\label{eq:3.5}
		\widehat{\mathbf{P}}[A_\rho'\cap \neg A_\rho]
		&\geq& \widehat{\mathbf{P}}[\mathscr{E}_R^n\cap A_\rho'\cap \neg A_\rho]=\sum_{e\in E\backslash B(\rho,R)} \widehat{\mathbf{P}}[\mathscr{E}_e^n\cap A_\rho'\cap \neg A_\rho]\nonumber\\
		&\geq&\sum_{e\in E\backslash B(\rho,R)} \widehat{\mathbf{P}}[\mathscr{E}_e^n\cap \Pi_eA_\rho'\cap \Pi_e\mathscr{P}_e]\nonumber\\
		&\stackrel{(\ref{eq:3.4})}{\geq}&p\sum_{e\in E\backslash B(\rho,R)} \widehat{\mathbf{P}}[\mathscr{E}_e^n\cap A_\rho'\cap \mathscr{P}_{e_n}]=p\widehat{\mathbf{P}}[\mathscr{E}_R^n\cap A_\rho'\cap \mathscr{P}_{e_n}]\nonumber\\
		&\geq&p\widehat{\mathbf{P}}[\mathscr{E}_R^n\cap A_\rho\cap \mathscr{P}_{e_n}]-p\varepsilon.
	\end{eqnarray}
	
	Since there are infinitely many heavy clusters a.s., Proposition \ref{prop:transience for heavy clusters} yields that $W$ is transient on the event that $C(\rho)$ is heavy, whence one can fix $n$ sufficiently large such that the probability that $C(\rho)$ is heavy and $e_n\in B(\rho,R)$ is smaller than $\varepsilon$, whence $\widehat{\mathbf{P}}[A_\rho\cap (\mathscr{E}^n-\mathscr{E}_R^n)]\leq \varepsilon$. Then by (\ref{eq:3.5}) one has for $n$ large 
	\begin{equation}\label{eq:3.6}
		\varepsilon\geq\widehat{\mathbf{P}}[A_\rho'\Delta  A_\rho]\geq \widehat{\mathbf{P}}[A_\rho'\cap \neg A_\rho]\geq p\widehat{\mathbf{P}}[\mathscr{E}^n\cap A_\rho\cap \mathscr{P}_{e_n}]-2p\varepsilon.
	\end{equation}
	
	Note by our choice of $r,e_n$, $\widehat{\mathbf{P}}[ A_\rho\cap \mathscr{P}_{e_0}]>0$. Conditioned on $A_\rho\cap \mathscr{P}_{e_0}$, transience of $W$ implies that there exists $m\leq0$ such that $W(m)$ is at graph distance $r$ to $e_0$ and $W(j)$ is at graph distance more than $r$ to $e_0$ for all $j<m$, in particular $W(j)$ is not incident to $e_0$, whence by choice of $r$, we have $\widehat{\mathbf{P}}[\mathscr{E}^m\cap A_\rho\cap \mathscr{P}_{e_m}]>0$. 
	
	Define $\mathscr{B}_m:=\cup_{x\in V}\mathscr{E}^m\cap A_\rho\cap \mathscr{P}_{e_m}\cap \{\rho=x \}$. Although event $\{\rho=x\}$ has zero probability under $\widehat{\mathbf{P}}$ for $x\in V\backslash \mathcal{O}$, the event $\mathscr{E}^m\cap A_\rho\cap \mathscr{P}_{e_m}\cap \{\rho=x \}$ is  well-defined just as the case $x\in\mathcal{O}$. 
	
	Notice  $\mathscr{B}_m$ is $\Gamma$-invariant and $\widehat{\mathbf{P}}[\mathscr{B}_m]=\widehat{\mathbf{P}}[\mathscr{E}^m\cap A_\rho\cap \mathscr{P}_{e_m}]$ for any $m\in\mathbb{Z}$.
	
	Let $\beta_{\xi,W}$ denote the law of $(e_n)_{n\in\mathbb{Z}}$ given $\xi,W$, then one has
	\begin{eqnarray}
		\widehat{\mathbf{P}}[\mathscr{B}_m]
		&=&\int_{2^E\times V^\mathbb{Z}\times E^\mathbb{Z}} \mathbf{1}_{[(\xi,W,(e_n)_{n\in\mathbb{Z}})\in\mathscr{B}_m]}d\widehat{\mathbf{P}}\nonumber\\
		&=&\int_{2^E\times V^\mathbb{Z}} \int_{E^\mathbb{Z}}\mathbf{1}_{[(\xi,W,(e_n)_{n\in\mathbb{Z}})\in\mathscr{B}_m]}d\beta_{\xi,W}((e_n)_{n\in\mathbb{Z}})d\mathbf{\Theta}(\xi,W)
	\end{eqnarray}
	
	Define  $F(\xi,W):=\int_{E^\mathbb{Z}}\mathbf{1}_{[(\xi,W,(e_n)_{n\in\mathbb{Z}})\in\mathscr{B}_m]}d\beta_{\xi,W}((e_n)_{n\in\mathbb{Z}})$. It is straightforward to check that $F$ is a $\Gamma$-invariant measurable function,  whence Corollary \ref{cor:two-sided biased rw} yields that $\widehat{\mathbf{P}}[\mathscr{B}_m]$ does not depend on $m$. Thus  $\widehat{\mathbf{P}}[\mathscr{E}^n\cap A_\rho\cap \mathscr{P}_{e_n}]=\widehat{\mathbf{P}}[\mathscr{B}_n]$  does not depend on $n$.  Hence when $\varepsilon>0$ is sufficiently small, (\ref{eq:3.6}) gives a contradiction. This completes the proof.
\end{proof}

Given a trajectory $w\in V^{\mathbb{Z}}$,
for a set $C\subset V$ and $m,n\in\mathbb{Z},m<n$, write 
\[
\alpha_m^n(C)(w):=\frac{1}{n-m}\sum_{k=m}^{n-1}\mathbf{1}_{\{w(k)\in C\}}
\]
and 
\[
\alpha(C)(w):=\lim_{n\rightarrow\infty}\frac{1}{n}\sum_{k=1}^{n}\mathbf{1}_{\{w(k)\in C \}}
\]
when the limit exists, for the frequency of visits to $C$ by the trajectory $w$ on $G$. We don't need to define $\alpha(C)$ when the limit does not exist due to the following 
generalization of Lemma 4.2 of \cite{LS1999a} in the quasi-transitive setting. 
\begin{lemma}\label{lem:freq}
	Suppose $\Gamma\subset \textnormal{Aut}(G)$ is quasi-transitive. Let $\widehat{\mathbf{P}}$ be the probability measure given in Corollary \ref{cor:srw}.  Then there is a $\Gamma$-invariant measurable function  $f:2^V\rightarrow[0,1]$ with the following property: $\widehat{\mathbf{P}}$-a.s. $\alpha(C)$ exists and is equal to $f(C)$ for every cluster $C$.  The function $f$ is called the frequency function. 
\end{lemma}
\begin{proof}
	The proof follows a similar strategy  as the one of  Lemma 4.2 of \cite{LS1999a}. Let $\lambda$ be a probability measure on $\{o_1,\ldots,o_L\}$ with $\lambda(\{o_i\})=c_i$, where these $c_i$'s come from Corollary \ref{cor:srw}. Let $\rho$ be sampled from $\lambda$. Let $\mathbf{P}_{\rho}=\sum_{i=1}^{L}c_i\mathbf{P}_{o_i}$ denote the law of simple random walk starting from random vertex $\rho$. Then we have $\widehat{\mathbf{P}}=\mathbf{P}_p\times \mathbf{P}_{\rho}$. 
	
	For every $\alpha\in[0,1]$, let
	\[
	\mathcal{Z}_\alpha:=\left\{ C\subset V: \lim_{n\rightarrow\infty}\alpha_0^n(C)=\alpha, \  \mathbf{P}_{\rho}\textnormal{-a.s.} \right\}
	\]
	Note this definition does not depend on the choice of basepoint $\rho$. 
	
	Define $f(C):=\alpha$ when $C\in\mathcal{Z}_\alpha$ for some $\alpha\in[0,1]$. If $C\notin \mathcal{Z}:=\cup_{\alpha\in[0,1]}\mathcal{Z}_{\alpha}$, put $f(C):=0$.  As shown in the proof of Lemma 4.2 of \cite{LS1999a},  $f$ is measurable and $\Gamma$-invariant and for $\widehat{\mathbf{P}}$-a.e.\ $(\xi,w)\in \{0,1\}^V\times V^{\mathbb{N}}$,
	\begin{equation}\label{eq:4.1}
		\lim_{n\rightarrow\infty} \max\{|\alpha_0^m(C)(w)-\alpha_0^k(C)(w)|: k,m\geq n,C \textnormal{ is a cluster of }\xi  \}=0.
	\end{equation}
	It remains to show that for $\mathbf{P}_p$-a.s.\ $\xi$, every cluster of $\xi$ is in $\mathcal{Z}$. Actually let $A:=\{\xi: C\in \mathcal{Z}, \forall \textnormal{  cluster }C \textnormal{ of }\xi \}$, it's easy to see $A\in\mathcal{F}_V$. 
	
	Corollary \ref{cor:two-sided srw} and \eqref{eq:4.1} then yield that for $\mathbf{P}_p\times\mathbf{P}_\rho^{\mathbb{Z}}$-a.s. $(\xi,\widehat{w})$, 
	\begin{equation}\label{eq:4.2}
		\lim_{n\rightarrow\infty} \max\{|\alpha_0^m(C)(\widehat{w})-\alpha_0^k(C)(\widehat{w})|: k,m\geq n,C \textnormal{ is a cluster of }\xi  \}=0.
	\end{equation}
	
	By the shift invariance of $\mathcal{T}$, 
	\begin{multline*}
	2\max\{|\alpha_0^{2n}(C)-\alpha_0^n(C)|: C \textnormal{ is a cluster} \}\\
	=\max\{|\alpha_n^{2n}(C)-\alpha_0^n(C)|: C \textnormal{ is a cluster} \}
	\end{multline*}
	has the same law as $\max\{|\alpha_0^{n}(C)-\alpha_{-n}^0(C)|: C \textnormal{ is a cluster} \}$.

	Therefore by \eqref{eq:4.2} we have  
	\begin{equation}\label{eq: conv in probability}
	\max\{|\alpha_0^{n}(C)-\alpha_{-n}^0(C)|: C \textnormal{ is a cluster} \}\rightarrow0\textnormal{ in probability}. 
	\end{equation}
	
		By \eqref{eq:4.2} we also have $\lim_{n\rightarrow\infty}\alpha_0^{n}(C)$ exists. Similar argument shows that $\lim_{n\rightarrow\infty}\alpha_{-n}^{0}(C)$ also exists. Combining with the above equation  \eqref{eq: conv in probability} one has that a.s.
		$\lim_{n\rightarrow\infty}\alpha_0^{n}(C)=\alpha(C)=\lim_{n\rightarrow\infty}\alpha_{-n}^{0}(C)$ for every cluster $C$.
	
	However given $\xi$ and a cluster $C$ of $\xi$, $\alpha_0^n(C)$ is independent of $\alpha_{-n}^0(C)$ conditioned on $w(0)$, but both converge to $\alpha(C)$. Therefore  $\alpha(C)$ is a $\mathbf{P}_p\times\mathbf{P}_\rho^{\mathbb{Z}}$-a.s. constant. This completes the proof. 
\end{proof}

\begin{proof}[Proof of Theorem \ref{thm:connectivity decay}]
	Write $\mathbf{P}=\mathbf{P}_p$ and $\Gamma=\textnormal{Aut}(G)$. If $\Gamma$ is unimodular, this is just Theorem 4.1 of \cite{LS1999a}. In the following we assume that $\Gamma$ is nonunimodular. 
	
	Since there is a.s.\ more than one infinite cluster, $N(p)=\infty$ a.s. Since light and heavy infinite clusters cannot coexist a.s., we consider two cases separately. If these infinitely many infinite clusters are light a.s., then $\inf\{\mathbf{P}(x\leftrightarrow y):x,y\in V\}=0$. Indeed, if $\inf\{\mathbf{P}(x\leftrightarrow y):x,y\in V\}=c_1>0$, fix some $o\in V$ and let $L_n:=L_n(o)$. Then $\mathbf{P}[o\leftrightarrow L_n]\geq \inf\{\mathbf{P}(x\leftrightarrow y):x,y\in V\}=c_1>0$. Notice $\{o\leftrightarrow L_n\}$ is decreasing. Thus $$\mathbf{P}[C(o) \textnormal{ is heavy}]\geq \mathbf{P}\left[\bigcap_{n=1}^{\infty}\{ o\leftrightarrow L_n\}\right] \geq \liminf_{n\rightarrow\infty}\mathbf{P}[o\leftrightarrow L_n]\geq c_1>0.$$ This contradicts the assumption that these infinite clusters are light. 
	
	Now in the following we assume that there are infinitely many heavy clusters a.s. 
	Let $f$ be the frequency function given in Lemma \ref{lem:freq}. 
	Since $f$ is $\Gamma$-invariant, for each $\alpha\in[0,1]$, $Q_\alpha:=\{C\subset 2^V: f(C)\leq \alpha\}$ is a $\Gamma$-invariant  property, where for a cluster $C$ we also use $C$ to denote its vertex set.
	
	Let $A_\alpha$ be the event that there exists an infinite cluster that satisfies the property $Q_\alpha$. If $\mathbf{P}(A_\alpha)>0$, then $\mathbf{P}(A_\alpha)=1$ by ergodicity of $\mathbf{P}$. 
	
	Let $c:=\inf\{\alpha\in[0,1]:\mathbf{P}(A_\alpha)>0 \}$. 
	By definition of $A_\alpha$ and $c$ we have $f(C)\geq c$ for all infinite clusters $C$ $\mathbf{P}$-a.s. On the other hand, for each $t>c$, $\mathbf{P}(A_t)>0$ and then $\mathbf{P}(A_t)=1$ by ergodicity. Thus $\mathbf{P}$-a.s. there exists an infinite cluster $C$ such that $f(C)\leq t$, i.e. $C$ satisfies $Q_t$. Since  there are infinitely many heavy clusters a.s., Theorem \ref{thm:heavy clusters are indis} then implies that all these infinite clusters satisfy property $Q_t$. Since this is true for arbitrary $t>c$, it follows that $\mathbf{P}$-a.s. 
	\[
	f(C)=c,\ \textnormal{ for every infinite cluster }C.
	\]
	The rest proof of the first conclusion is almost the same as the one of Theorem 4.1 of \cite{LS1999a} except we use Lemma \ref{lem:freq} to get the a.s. equality $f(C)=\alpha(C)$ in this  quasi-transitive nonunimodular setting. 
	
	And  deriving $p_u=\overline{p}_{\textnormal{conn}}$ from the first conclusion is easy so we leave it to the reader.
\end{proof}

\begin{remark}
	At $p=p_u$, $ \inf_{x\in V}\mathbb{P}_{p_u}(o\leftrightarrow x)$ can be equal to zero or be positive. For example let $G$ be a nonamenable planar quasi-transitive graph with one end. Then $0<p_c<p_u<1$ (see Theorem 8.24 of \cite{LP2016}) and there is a unique infinite cluster $\mathbb{P}_{p_u}$-a.s. In this case by Harris-FKG inequality we know $\inf_{x\in V}\mathbb{P}_{p_u}(o\leftrightarrow x)>0$. Other examples include $\mathbb{T}_b$ with $b\geq3$ or $\mathbb{Z}^2*\mathbb{Z}_2$, where $*$ means the free product. They are transitive graphs with infinitely many ends and hence $p_u=1$. 
	
	In the following we will use $X\times Y$ to denote the Cartesian product of graphs $X$ and $Y$. 
	Consider $\mathbb{T}_b\times \mathbb{Z}$ where $\mathbb{T}_b$ is a regular tree with degree $b\geq3$. We know $0<p_c<p_u<1$ from \cite{Tom2017},\cite{BB1999}. Schonmann \cite{S1999b} and Peres \cite{P2000} showed that at $p=p_u$ there are infinitely many infinite clusters a.s. 
	Now from Theorem \ref{thm:connectivity decay} we know $ \inf_{x\in V}\mathbb{P}_{p_u}(o\leftrightarrow x)=0$ in this case. For the case $G=\mathbb{T}_1\times\cdots\times\mathbb{T}_n,n\geq 2$, where $\mathbb{T}_i$ are regular trees with degree at least $3$,  Theorem \ref{thm:connectivity decay} implies that $ \inf_{x\in V}\mathbb{P}_{p_u}(o\leftrightarrow x)=0$. In this case Hutchcroft \cite[Question 1.9]{Tom2017b} conjectured that $\mathbb{P}_{p_u}(o\leftrightarrow x)$ even decays exponentially in the graph distance  $d(o,x)$. 
\end{remark}

\section{Examples and questions}

\subsection{Examples 1}
Consider the regular tree $\mathbb{T}_b$ with degree $b\geq 3$. Let $\xi$ be a fixed end of $\mathbb{T}_b$ and $\Gamma_\xi\subset\textnormal{Aut}(\mathbb{T}_b)$ be the subgroup that fixes the end $\xi$. Then $\Gamma_\xi$ is transitive and nonunimodular. And $p_h(G,\Gamma_\xi)=1$. This can be seen by direct calculation (comparing to a branching process) or by Corollary \ref{cor:p_h equal limit of p_c}. Indeed since $p_u=1$, $\lim_{n\rightarrow\infty}p_c(G_n)\leq p_u$ holds and then $p_h(G,\Gamma_\xi)=\lim_{n\rightarrow\infty}p_c(G_n)$. Since $G_n$ are finite unions of levels, all its connected components are finite, whence $p_c(G_n)=1$ and then $p_h(G,\Gamma_\xi)=\lim_{n\rightarrow\infty}p_c(G_n)=1$.  More information about this example can be found in Section 8.1 of \cite{Tom2017}.  This gives an example with $p_c<p_h(G,\Gamma)=p_u=1$. 

Now consider the Cartesian product $G:=\mathbb{T}_b\times \mathbb{Z}^d$ with  $b\geq 3,d\geq1$. Let $\Gamma=\Gamma_\xi\times\textnormal{Aut}(\mathbb{Z}^d)$. Corollary 5.8 of \cite{Timar2006} implies that
$p_h(G,\Gamma)=p_u$ since the subgraph induced by any finite unions of levels is amenable. Peres \cite{P2000} showed that $N(p_h(G,\Gamma))=\infty$ a.s. Proposition \ref{prop:light at p_h} yields all these infinite clusters are light. This gives an example that $p_c<p_h(G,\Gamma)=p_u<1$ and $N(p_h(G,\Gamma))=\infty$. Here $p_c<p_h(G,\Gamma)$ is due to \cite{Tom2017}. 

\begin{question}
	Is there a graph $G$ and $\Gamma\subset\textnormal{Aut}(G)$ such that $\Gamma$ is quasi-transitive and nonunimodular,  $p_c<p_h(G,\Gamma)=p_u<1$ and $N(p_h(G,\Gamma))=1$ a.s.?
\end{question}

It was asked in \cite{HPS1999} whether there is any explicit example of transitive graph satisfying $p_c<p_h<p_u<1$. The first such examples are provided in \cite{Tom2017}; see section 8.2 there. The following examples are also devoted to this question. \textbf{Examples 2} are slightly simpler in the sense one don't need to consider anisotropic percolation. However \textbf{Examples 2} exhibit $p_h(G,\Gamma)<p_u(G)$ with  proper subgroup $\Gamma$. One might modify them to obtain examples with $\Gamma=\textnormal{Aut}(G)$. We will not do that and instead we point out \textbf{Examples 3} satisfying the restriction $\Gamma=\textnormal{Aut}(G)$.

\subsection{Examples 2}
To be consistent with the notation in \cite{LP2016}, we consider regular trees $\mathbb{T}_{b+1}$ with $b\geq 3$. Suppose $n_1,n_2$ are two positive integers such that $n_1+n_2=b$. We define a $(1,n_1,n_2)$-\textbf{orientation} of $\mathbb{T}_{b+1}$ like the  $n_1=1,n_2=2,b=3$ case in \cite{Tom2017}. To be precise, give a partial orientation of the edge set of $\mathbb{T}_{b+1}$ such that 
every vertex is incident to exactly one unoriented edge, has $n_1$ oriented edges emanating from it, and has $n_2$ oriented edges pointing into it. From now on we fix such an orientation.

Denote by $\Gamma_{(1,n_1,n_2)}\subset\textnormal{Aut}(\mathbb{T}_{b+1})$ the subgroup of automorphisms that preserve this orientation. It is easy to see that this subgroup $\Gamma_{(1,n_1,n_2)}$ acts transitively on $\mathbb{T}_{b+1}$. Moreover, if $q:=\frac{n_2}{n_1}\neq 1$, then $\Gamma_{(1,n_1,n_2)}$ is nonunimodular. Indeed, define $h(u,v)$ to be the \textbf{height  difference} as in \cite{Tom2017}:
for $u,v\in\mathbb{T}_{b+1}$ there is a unique simple path $r$ connecting them, suppose there are $m_1$ edges on $r$ that are crossed in the forward direction when moving from $u$ to $v$ along $r$ and $m_2$ edges that are crossed in the opposite direction, then $h(u,v):=m_1-m_2$.  Using Lemma \ref{lem:haar} it's easy to see $\Delta_{\Gamma}(u,v)=\frac{|\Gamma_vu|}{|\Gamma_uv|}=q^{h(u,v)}$ for $u\sim v$, where $\Gamma:=\Gamma_{(1,n_1,n_2)}$. Then by cocycle identity  $\Delta_{\Gamma}(u,v)=q^{h(u,v)}$ holds for all pairs $u,v\in\mathbb{T}_{b+1}$.

One has  $p_h(\mathbb{T}_4,\Gamma_{(1,1,2)})=\frac{2\sqrt{2}+1-\sqrt{4\sqrt{2}-3}}{6}$  by 
Proposition 8.1 of \cite{Tom2017} and  formula (8.1) there. 
The following proposition is a slight generalization of this result.
\begin{proposition}\label{prop:exact value of p_h}
	Suppose positive integers $n_1,n_2$ satisfy  $n_1+n_2=b$, then 
	\begin{equation}\label{eq:exact value of p_h}
		p_h(\mathbb{T}_{b+1},\Gamma_{(1,n_1,n_2)})= \frac{1+2\sqrt{n_1n_2}-\sqrt{(2\sqrt{n_1n_2}+1)^2-4(n_1+n_2)}}{2(n_1+n_2)}.
	\end{equation}
	
\end{proposition}
We will not provide its proof but point out two ways to do it in the following two remarks.
\begin{remark}\label{rem:p_t=p_h}
	If one defines the \textbf{tiltability threshold} $p_t$ as in \cite{Tom2017}, using the same method of obtaining formula (8.1) in \cite{Tom2017}, one can calculate the exact value of $p_t$ and it is just the right-hand side of \eqref{eq:exact value of p_h}. Then one can use similar argument as in the proof of Proposition 8.1  of \cite{Tom2017} to obtain $p_h(\mathbb{T}_{b+1},\Gamma_{(1,n_1,n_2)})=p_t(\mathbb{T}_{b+1},\Gamma_{(1,n_1,n_2)})$. 
\end{remark}
\begin{remark}\label{rem:simple method of calculate p_h}
	Another way to prove Proposition \ref{prop:exact value of p_h} is using Corollary \ref{cor:p_h equal limit of p_c}. Notice that the finite union of consecutive levels $G_n$ are infinitely many copies of periodic trees $H_n$. These periodic trees are directed covers (\cite[Section 3.3]{LP2016})  of certain finite oriented graphs $D_n$. Notice $p_c(H_n)^{-1}=\textnormal{br}(H_n)=\textnormal{gr}(H_n)$ for periodic trees. Moreover let $A_n$ denote the adjacency matrix of $D_n$, then $\textnormal{gr}(H_n)=\lambda_*(A_n)$ (see the discussion on pp 83-84 of \cite{LP2016}), where $\lambda_*(A_n)$ denotes the largest positive eigenvalue  of the matrix $A_n$. Then  the reciprocal of $p_h(\mathbb{T}_{b+1},\Gamma_{(1,n_1,n_2)})$ equals the limit of $\lambda_*(A_n)$. Calculating the limit of $\lambda_*(A_n)$ then gives Proposition \ref{prop:exact value of p_h}. The calculation is a little bit long and we omit it here.

	Russell Lyons pointed out to me that the adjacency matrices  $A_n$ are parts of a block Toeplitz matrix $A$ and that \cite{EF2000} may be relevant. Indeed the formula from \cite{EF2000} does give exactly the reciprocal of \eqref{eq:exact value of p_h}. However, the theorem in \cite{EF2000}  does not include our matrix $A$. One might expect to extend the theorem in \cite{EF2000} in order to have a simple way to find the limit of $\lambda_*(A_n)$. 
\end{remark}

Now we consider $\mathbb{T}_{b+1}\times\mathbb{Z}$ and $\Gamma_{b}:=\Gamma_{(1,n_1,n_2)}\times\textnormal{Aut}(\mathbb{Z})$. From Theorem 6.10, Proposition 7.35 and Theorem 7.37 of \cite{LP2016} one has the following lower bound for $p_u$:
\begin{equation}\label{eq:lower bound for p_u}
	p_u(\mathbb{T}_{b+1}\times\mathbb{Z})\geq \frac{1}{\textnormal{cogr}(\mathbb{T}_{b+1}\times\mathbb{Z})}=\frac{1}{\sqrt{b}+1+\sqrt{2\sqrt{b}-1}}. 
\end{equation}

Note $p_h(\mathbb{T}_{b+1}\times\mathbb{Z},\Gamma_b)\leq p_h(\mathbb{T}_{b+1},\Gamma_{(1,n_1,n_2)})$. Simple calculation shows that when $n_1\geq2,n_2\geq2$ and $b$ large enough, the value of $ p_h(\mathbb{T}_{b+1},\Gamma_{(1,n_1,n_2)})$ (Proposition \ref{prop:exact value of p_h}) is strictly less than the above lower bound of $p_u$, whence such graphs are explicit examples exhibiting
$p_c(\mathbb{T}_{b+1}\times\mathbb{Z})<p_h(\mathbb{T}_{b+1}\times\mathbb{Z},\Gamma_b)<p_u(\mathbb{T}_{b+1}\times\mathbb{Z})<1$. However here $\Gamma_b$ is not the whole automorphism group of $\mathbb{T}_{b+1}\times\mathbb{Z}$.

\subsection{Examples 3}
The following family of examples are motivated by \cite{HPS1999}.
Recall for a quasi-transitive graph $G$, $p_h:=p_h(G,\textnormal{Aut}(G))$. 

\begin{definition}[Definition 1.3 in \cite{Sa1960}]
	A graph $G$ is called \textbf{prime} w.r.t. Cartesian product if $G$ is non-trivial (not the graph $U$ with a single vertex and no edge) and if $G\cong Y\times Z$ then either $Y\cong U$ or $Z\cong U$, where $A\cong B$ means that graph $A$ is isomorphic to $B$. Two distinct graphs $G,G'$ are called \textbf{relatively prime} if $G\cong X\times Z$ and $G'\cong Y\times Z$ implies that $Z\cong U$. 
\end{definition}

Now fix $G_0$ to be a  nonunimodular transitive graph and that is relatively prime to regular trees (for example, $G_0$ can be the grand-parent graph). Then for any regular tree $\mathbb{T}_k$ by Corollary 3.2 of \cite{Sa1960}   
$\textnormal{Aut}(G_0\times \mathbb{T}_k)=\textnormal{Aut}(G_0)\times\textnormal{Aut}(\mathbb{T}_k)$, whence it is nonunimodular. 

For $k$ large enough one has $p_h(G_0\times \mathbb{T}_k)<p_u(G_0\times \mathbb{T}_k)$ (see the inequality (4.9.2) on page 87 of \cite{HPS1999}). By Hutchcroft \cite{Tom2017} and the fact that $\textnormal{Aut}(G_0\times \mathbb{T}_k)$ is nonunimodular and transitive, one has $p_c(G_0\times \mathbb{T}_k)<p_h(G_0\times \mathbb{T}_k)$, whence we get another family of graphs exhibiting $p_c<p_h<p_u<1$. 
\\

Last but not least, Theorem \ref{thm:heavy clusters are indis} and Theorem \ref{thm:connectivity decay} are only proved for Bernoulli percolation. However,  corresponding theorems in \cite{LS1999a} are proved to hold for general insertion tolerant percolation process.

Remark \ref{rem:remark after prop heavy transient} implies for insertion-and-deletion tolerant percolation process the following weaker conclusion hold:  there exists some heavy cluster that is transient for the ``square-root biased" random walk. Notice in the proof  of Theorem \ref{thm:heavy clusters are indis} we need that with positive probability the cluster $C(\rho)$ is heavy, transient for  the ``square-root biased" random walk and has some pivotal edges. However the weak conclusion and Lemma  \ref{lem:pivotal edge} do not guarantee the existence of a pivotal edge for $C(\rho)$.

In Tim\'{a}r's proof of Proposition \ref{prop:finite levels},  deletion-tolerance was used in the proof of Lemma 5.2 and Lemma 5.3 in \cite{Timar2006}. Lemma 5.2 can be extended to percolation processes with just insertion-tolerance property but we do not know whether Lemma 5.3 can be extended to such percolation processes. 

\begin{question}
	Does Theorem \ref{thm:heavy clusters are indis} hold if one just assume $\Gamma$-invariance and insertion-and-deletion tolerance? What if just $\Gamma$-invariance and insertion-tolerance?
\end{question}

\section*{Acknowledgements} We thank Russell Lyons for many helpful discussions, detailed comments on a first draft and financial support by the National Science
Foundation under grant DMS-1612363. We also thank \'{A}d\'{a}m Tim\'{a}r for comments on the manuscript. We thank the referee for the  detailed and  helpful comments which improve the quality of writing a lot.  

\bibliographystyle{imsart-nameyear}


\end{document}